\documentclass[reqno,12pt]{amsart}
\usepackage{amsfonts}
\usepackage{bbm}
\usepackage{} 
\numberwithin{equation}{section}
\usepackage{indentfirst}
 \usepackage{color}
\usepackage{amssymb}
\usepackage{mathrsfs}
\usepackage{xy}
\xyoption{all}
\def\Ext{\mbox{\rm Ext}\,} \def\Hom{\mbox{\rm Hom}} \def\dim{\mbox{\rm dim}\,} \def\Iso{\mbox{\rm Iso}\,}\def\Ind{\mbox{\rm Ind}\,}
\def\lr#1{\langle #1\rangle}    
\def\Ker{\mbox{\rm Ker}\,}   \def\im{\mbox{\rm Im}\,} \def\Coker{\mbox{\rm Coker}\,}
\def\End{\mbox{\rm End}\,}\def\tw{\mbox{\rm tw}\,}\def\id{\mbox{\rm id}\,}

\def\rad{\mbox{\rm rad}\,}\def\M{\mathcal{M}}
\def\Aut{\mbox{\rm Aut}\,}\def\Dim{\mbox{\rm \textbf{dim}}\,}\def\A{\mathcal{A}\,} \def\H{\mathcal{H}\,}
\def\P{\mathscr{P}\,}\def\X{\mathbb X}

\def\ZZ{\mathbb Z}

\theoremstyle{plain} 
\newtheorem{theorem}{\bf Theorem}[section]
\newtheorem{lemma}[theorem]{\bf Lemma}
\newtheorem{corollary}[theorem]{\bf Corollary}
\newtheorem{proposition}[theorem]{\bf Proposition}

\theoremstyle{definition} 
\newtheorem{definition}[theorem]{\bf Definition}
\newtheorem{remark}[theorem]{\bf Remark}
\newtheorem{example}[theorem]{\bf Example}

\newcommand{\bt}{\begin{theorem}}
\newcommand{\et}{\end{theorem}}
\newcommand{\bl}{\begin{lemma}}
\newcommand{\el}{\end{lemma}}
\newcommand{\bd}{\begin{definition}}
\newcommand{\ed}{\end{definition}}
\newcommand{\bc}{\begin{corollary}}
\newcommand{\ec}{\end{corollary}}
\newcommand{\bp}{\begin{proof}}
\newcommand{\ep}{\end{proof}}
\newcommand{\bx}{\begin{example}}
\newcommand{\ex}{\end{example}}
\newcommand{\br}{\begin{remark}}
\newcommand{\er}{\end{remark}}
\newcommand{\be}{\begin{equation}}
\newcommand{\ee}{\end{equation}}
\newcommand{\ba}{\begin{align}}
\newcommand{\ea}{\end{align}}
\newcommand{\bn}{\begin{enumerate}}
\newcommand{\en}{\end{enumerate}}
\newcommand{\bcs}{\begin{cases}}
\newcommand{\ecs}{\end{cases}}

\makeatletter
\renewcommand{\section}{\@startsection{section}{1}{0mm}
  {-\baselineskip}{0.5\baselineskip}{\bf\leftline}}
\makeatother

\begin{document}

\title[Quantum cluster characters of Hall algebras revisited]{Quantum cluster characters \\ of Hall algebras revisited} 

\author{Changjian Fu, Liangang Peng and Haicheng Zhang$^*$}
\address{Department of Mathematics\\ SiChuan University\\ Chengdu 610064, P.~R.~China}
\email{changjianfu@scu.edu.cn (C. Fu)}
\address{Department of Mathematics\\ SiChuan University\\ Chengdu 610064, P.~R.~China} \email{penglg@scu.edu.cn
(L. Peng)}
\address{Institute of Mathematics, School of Mathematical Sciences, Nanjing Normal University,
Nanjing 210023, P.~R.~China}
\email{zhanghc@njnu.edu.cn (H. Zhang)}


\thanks{$*$~Corresponding author.}


\begin{abstract}
Let $Q$ be a finite acyclic valued quiver. We define a bialgebra structure and an integration map on the Hall algebra associated to the morphism category of projective representations of $Q$. As an application, we recover the surjective homomorphism defined in \cite{DXZ}, which realizes the principal coefficient quantum cluster algebra $\A_q(Q)$ as a sub-quotient of the Hall algebra of morphisms. Moreover, we also recover the quantum Caldero--Chapoton formula, as well as some multiplication formulas between quantum Caldero--Chapoton characters.
\end{abstract}

\maketitle

\section{Introduction}
The Hall algebra of a finitary abelian category is defined to be an associative algebra with a basis indexed by the isomorphism classes of objects and with a multiplication which encodes information about extensions of objects (cf. \cite{R90,Sc}). A typical example of such categories is provided by the category of finite dimensional representations of a finite acyclic quiver over a finite field. For a Dynkin quiver $Q$, Ringel~\cite{R90a} realized the positive part of the quantum group associated to $Q$ via the so-called Ringel--Hall algebra of $Q$. Later on, for any hereditary algebra $A$, Green \cite{Gr95} introduced a bialgebra structure on the Ringel--Hall algebra of $A$,
and he showed that the composition subalgebra generated by simple modules
provides a realization of the positive part of the corresponding quantum group.

Cluster algebras were introduced by Fomin and Zelevinsky~\cite{FZ} with the aim to set up a combinatorial framework for the study of total positivity  in algebraic groups and canonical bases in quantum groups. The quantum versions of cluster algebras, called the quantum cluster algebras, were later introduced by Berenstein and Zelevinsky~\cite{BZ05}.  Acyclic (quantum) cluster algebras associated with acyclic valued quivers are an important class of (quantum) cluster algebras, which has a close relation with the representation theory of acyclic valued quivers, see~\cite{BMRRT,CC,CK2005, CK2,Qin, Rupel1, Rupel2} for instance. Among others, the combinatorial structure of an acyclic cluster algebra $\mathcal{A}(Q)$ can be categorified by the combinatorial structure of the category of representations of $Q$~({cf.}~\cite{BMRRT}). Caldero and Chapoton~\cite{CC} proved that the cluster algebra $\mathcal{A}(Q)$ associated with a Dynkin quiver $Q$ can be recovered from the data of the corresponding quiver representation category. Caldero and Keller~\cite{CK2005} further proved the cluster multiplication formulas for $\mathcal{A}(Q)$.

Let $Q$ be a finite acyclic valued quiver and $\mathcal{A}$ be the category of finite dimensional representations of $Q$. The work of Caldero and Keller~\cite{CK2005} suggested that there should be a deep connection between the quantum cluster algebra $\mathcal{A}_q(Q)$ (with coefficients) and  the dual Hall algebra $\mathcal{H}(\mathcal{A})$. In~\cite{DX}, Ding and Xu showed that the comultiplication of $\mathcal{H}(\mathcal{A})$ implies the cluster multiplication formulas in $\mathcal{A}_q(Q)$~({cf.} also ~\cite{Fei}).
The expected connection has been established by Berenstein and Rupel~\cite{BR}. In particular, Berenstein and Rupel~\cite{BR} proved that there is a homomorphism of algebras from $\mathcal{H}(\mathcal{A})$ to the quantum torus  associated with the quantum cluster algebra $\mathcal{A}_q(Q)$, which maps the indecomposable rigid modules to the non-initial quantum cluster variables of $\mathcal{A}_q(Q)$. Such a homomorphism is called a quantum cluster character of the Hall algebra. The homomorphism and its generalization have been further investigated in~\cite{CDX,Fu, Rupel3}. Let us emphasize that the bialgebra structure of $\mathcal{H}(\mathcal{A})$ has also played a central role in~\cite{BR, CDX, Fei, Fu}.

In general, the initial quantum cluster variables do not belong to the image of a quantum cluster character of the Hall algebra $\mathcal{H}(\mathcal{A})$. In order to overcome this shortcoming, Ding, Xu and Zhang~\cite{DXZ} introduced the (localized) Hall algebra $\mathcal{MH}(\mathcal{A})$ of morphisms associated to the representation category of $Q$, and established a homomorphism of algebras from a twisted version $\mathcal{MH}_\Lambda(\mathcal{A})$ of $\mathcal{MH}(\mathcal{A})$ to the quantum torus  associated with the quantum cluster algebra $\mathcal{A}_q(Q)$ with principal coefficients. Along this way, the quantum cluster algebra $\mathcal{A}_q(Q)$ has been realized as a subquotient of the Hall algebra $\mathcal{MH}_\Lambda(\mathcal{A})$.

The aim of this paper is to pursue a bialgebra approach to understand the quantum cluster character introduced in~\cite{DXZ}.
To do this, we introduce a comultiplication $\Delta$ for $\mathcal{MH}(\mathcal{A})$ and prove that  $\mathcal{MH}(\mathcal{A})$ is a bialgebra without counit. We also introduce an integration map $\int$ from $\mathcal{MH}(\mathcal{A})$ to a suitable torus.
Our main result yields an explicit and simple factorization of the quantum cluster character $\Psi$ in~\cite{DXZ} via the comultiplication $\Delta$ and integration map $\int$. This gives an alternative way to recover the quantum cluster character $\Psi$. Compared with \cite{DXZ}, we obtain the quantum cluster character $\Psi$ without using cluster multiplication formulas, on the contrary, cluster multiplication formulas (cf. \cite{DX,Fei,DSC,DXZ}) are the outcome of our main result. Moreover, we also recover the quantum Caldero--Chapoton formula, and present its intrinsic interpretation via the Hall algebra of morphisms.
As an appendix, we also show how to apply a certain subalgebra of the derived Hall algebra of the bounded derived category $\mathcal {D}^b(\mathcal{A})$ to realize the whole acyclic quantum cluster algebras. By the way, the three-step homomorphism approach used in this paper has been further developed in \cite{CDZ} to deduce the quantum version of the cluster multiplication formulas in the classical cluster algebra proved by Caldero-Keller \cite{CK2005} for finite type, Hubery \cite{Hubery1} for affine type and Xiao-Xu \cite{XX} for acyclic quivers.

The paper is organized as follows: In Section $2$ we recollect the definitions and properties of the morphism category and the construction of the Hall algebra of morphisms.  We introduce a bialgebra structure for $\mathcal{MH}(\mathcal{A})$ in Section $3$ and an integration map in Section $4$. After certain computations involving compatible pairs in~ Section $5$, we prove the main result in Section $6$. As a byproduct, we give a categorical interpretation of $g$-vectors via the morphism category in Section $7$. In Section $8$, we generalize the idea of  Section $6$ to show that a certain subalgebra of the derived Hall algebra can also be used to realize acyclic quantum cluster algebras.

~~~~~~~~~~~~~~~~~~~~~~
Let us fix some notations used throughout the paper. For a finite set $S$, we denote by $|S|$ its cardinality. Let $k=\mathbb{F}_q$ be a finite field with $q$ elements, and set $v=\sqrt{q}$. Let $\ZZ[v,v^{-1}]$ be the ring of integral Laurent polynomials. Let $A$ be a finite dimensional hereditary $k$-algebra, and denote by $\A$ the category of finite dimensional left $A$-modules; let $\mathscr{P}=\mathscr{P}_{\A}\subset\A$ be the subcategory consisting of projective objects.  For an essentially small exact or triangulated category $\mathcal {E}$, the Grothendieck group of $\mathcal {E}$ and the set of isomorphism classes $[X]$ of objects in $\mathcal {E}$ are denoted by $K(\mathcal {E})$ and $\Iso(\mathcal {E})$, respectively; let $\Ind(\mathcal {E})$ be a complete set of indecomposable objects in $\mathcal {E}$. For each object $M$ in $\mathcal {E}$, its automorphism group and image in $K(\mathcal {E})$ are denoted by $\Aut(M)$ and $\hat{M}$, respectively; we set $a_M=|\Aut(M)|$ and denote by $aM$ the direct sum of $a$ copies of $M$ for a positive integer $a$. For a module $M\in\A$, we also use $\Dim M$  to denote its dimension vector. We always assume that all the vectors are column vectors, and all tensor products are taken over $\ZZ[v,v^{-1}]$.

\section{Preliminaries}

\subsection{Hall algebras}
Given objects $L,M,N \in \mathcal{A}$, let $\Ext_\mathcal{A}^1(M,N)_L \subset \Ext_\mathcal{A}^1(M,N)$ be the subset consisting of those equivalence classes of short exact sequences with middle term isomorphic to $L$.
\begin{definition}\label{Hall algebra of abelian category}
The \emph{Hall algebra} $\mathcal {H}(\mathcal{A})$ of $\mathcal{A}$ is the free $\ZZ[v,v^{-1}]$-module with basis elements $[M] \in \Iso(\mathcal{A})$, and with the multiplication defined by
\[[M] \diamond [N] = \sum\limits_{[L] \in \Iso(\mathcal{A})} {\frac{{|\Ext_\mathcal{A}^1{{(M,N)}_L}|}}{{|\Hom_\mathcal{A}(M,N)|}}} [L].\]
\end{definition}
\begin{remark}
Given objects $L,M,N\in \A$, set
$$F_{MN}^L:=|\{N'\subset L~|~N'\cong N, L/N'\cong M\}|.$$
By the Riedtmann--Peng formula \cite{Riedtmann,Peng},
$$F_{MN}^{L}=\frac{|\Ext^1_{\A}(M,N)_{L}|}{|\Hom_{\A}(M,N)|}\frac{a_{L}}{a_{M}a_{N}}.$$ Thus, in terms of the rescaled generators $[[M]]=\frac{[M]}{a_M}$, the product takes the form
$$[[M]]\diamond [[N]]= \sum\limits_{[L] \in {\rm Iso}(\mathcal{A})}F_{MN}^L[[L]],$$
which is the definition used, for example, in \cite{R90a,Sc}. The associativity of Hall algebras amounts to the following identity
\begin{equation}\label{jiehe}\sum\limits_{[M]}F_{XY}^MF_{MZ}^L=\sum\limits_{[N] }F_{XN}^LF_{YZ}^N=:F_{XYZ}^L,\end{equation} for any objects $L,X,Y,Z\in\A$.
\end{remark}

Given objects $M,N \in \mathcal{A}$, one defines \begin{equation}\label{Euler form}\lr{M,N}:=\dim_k\Hom_{\A}(M,N)-\dim_k\Ext^1_{\A}(M,N),\end{equation}
which descends to give a bilinear form
\begin{equation}\label{adeuler}\lr{\cdot ,\cdot }: K(\mathcal{A})\times K(\mathcal{A})\longrightarrow \mathbb{Z},\end{equation}
called the \emph{Euler form} of $\mathcal{A}$.
We also consider the \emph{symmetric Euler form}
$$(\cdot ,\cdot ): K(\mathcal{A})\times K(\mathcal{A})\longrightarrow \mathbb{Z}$$ defined by $(\alpha,\beta)=\lr{\alpha,\beta}+\lr{\beta,\alpha}$ for all $\alpha,\beta \in K(\mathcal{A})$.

The \emph{twisted Hall algebra} $\H_{\tw}(\A)$ is the same module as $\H(\A)$ but with the twisted multiplication defined by
$$[M]*[N]=q^{\lr{M,N}}[M]\diamond[N].$$

\subsection{Morphism categories}\label{morphism}
Let $C_2(\A)$ be the category whose objects are the morphisms $\xymatrix{M_{-1}\ar[r]^f&M_0}$ in $\A$, and each morphism from $\xymatrix{M_{-1}\ar[r]^f&M_0}$ to $\xymatrix{N_{-1}\ar[r]^g&N_0}$ is a pair $(u,v)$ of morphisms in $\A$ such that the following diagram
$$\xymatrix{M_{-1}\ar[r]^f\ar[d]_u&M_0\ar[d]^v\\
N_{-1}\ar[r]^g&N_0}$$ is commutative. In what follows, we also write $M_\bullet$ as a morphism $\xymatrix{M_{-1}\ar[r]^f&M_0.}$ Let $C_2(\P)\subset C_2(\A)$ be the extension-closed subcategory consisting of the morphisms in $\P$.

For each object $P\in\P$, define two objects in $C_2(\P)$
\begin{equation}K_P:=\xymatrix{P\ar[r]^1&P}~~\text{and}~~Z_P:=\xymatrix{P\ar[r]&0.}\end{equation}
For each object $M\in\A$, fixing a minimal projective resolution\footnote{The notations $P_M$ and $\Omega_M$ will be used throughout the paper.} \begin{equation}\label{mpr}
\xymatrix{0\ar[r]&\Omega_M\ar[r]^{\delta_M}&P_M\ar[r]&M\ar[r]&0,}
\end{equation} by \cite{DXZ} we define an object in $C_2(\P)$
\begin{equation}\label{cm}
C_M:=\xymatrix{\Omega_M\ar[r]^{\delta_M}&P_M.}
\end{equation}
The indecomposable objects of $C_2(\P)$ are characterized in the following

\begin{lemma}\cite[Proposition 2.4]{DXZ}\label{fenjie}
Each object $M_\bullet$ in $C_2(\P)$ has a direct sum decomposition
$$M_\bullet=K_P\oplus Z_Q\oplus C_M$$
for some $P, Q\in\P$ and $M\in\A$. Moreover, the objects $P, Q$ and $M$ are uniquely determined up to isomorphism.
\end{lemma}

\begin{lemma}\cite[Corollaries 3.1 and 3.2]{Bau}\label{proinj}
The objects $C_P$ and $K_P$, where $P\in \Ind(\P)$, provide a complete set of indecomposable projective objects in $C_2(\mathscr{P})$; and the objects $Z_P$ and $K_P$ provide a complete set of indecomposable injective objects. Moreover, all $K_P$ are exactly all the indecomposable projective-injective objects.
\end{lemma}

For each $M_\bullet\in C_2(\P)$, we have a projective resolution and an injective resolution of $M_\bullet$ in the following
\begin{lemma}\cite[Proposition 3.2]{Bau}\label{resol}
For each $M_\bullet\in C_2(\P)$, we have the following short exact sequences
\begin{flalign}
&0\longrightarrow C_{M_{-1}}\longrightarrow C_{M_0}\oplus K_{M_{-1}}\longrightarrow M_{\bullet}\longrightarrow0;\label{tfj}\\
&0\longrightarrow M_{\bullet}\longrightarrow Z_{M_{-1}}\oplus K_{M_0}\longrightarrow Z_{M_0}\longrightarrow0.
\end{flalign}
\end{lemma}

By Lemma \ref{resol}, we know that the global dimension of $C_2(\P)$ is equal to one. Thus, for any objects $M_\bullet,N_\bullet\in C_2(\P)$, we define
\begin{equation}
\lr{M_\bullet,N_\bullet}:=\dim_k\Hom_{C_2(\P)}(M_\bullet,N_\bullet)-\dim_k\Ext_{C_2(\P)}^1(M_\bullet,N_\bullet).
\end{equation}
It induces a bilinear form on the Grothendieck group of $C_2(\P)$. We use the same notation as the Euler form of $\A$, since this should not cause confusion by the context.
By \cite[Lemma 2.9]{DXZ}, for any $M_\bullet,N_\bullet\in C_2(\P)$,
\begin{equation*}\lr{M_\bullet,N_\bullet}=\dim_k\Hom_{\A}(M_0,N_0)+\dim_k\Hom_{\A}(M_{-1},N_{-1})-\dim_k\Hom_{\A}(M_{-1},N_0).\end{equation*}

\subsection{Hall algebras of morphisms}
Let $\H(C_2(\A))$ be the Hall algebra of $C_2(\A)$ as defined in Definition \ref{Hall algebra of abelian category}. Let $\H(C_2(\P))$ be the submodule of $\H(C_2(\A))$ spanned by the isomorphism classes of objects in $C_2(\P)$. Since $C_2(\P)$ is closed under extensions, $\H(C_2(\P))$ is a subalgebra of the Hall algebra $\H(C_2(\A))$.
Define $\H_{\tw}(C_2(\P))$ to be the same module as $\H(C_2(\P))$, but with the twisted multiplication
$$[M_\bullet]\ast[N_\bullet]=q^{\lr{M_\bullet,N_\bullet}}[M_\bullet]\diamond[N_\bullet].$$
For any $M_\bullet\in C_2(\P)$ and $P\in\P$,
by \cite[Lemma 3.1]{DXZ}, we have that in $\H_{\tw}(C_2(\P))$
\begin{equation}
[K_P]\ast[M_\bullet]=[M_\bullet]\ast[K_P]=[K_P\oplus M_\bullet].\end{equation}
Define the {\em localized Hall algebra} $\M\H(\A)$ to be the localization of $\H_{\tw}(C_2(\P))$ with respect to all elements $[K_P]$. In symbols,
$$\M\H(\A):=\H_{\tw}(C_2(\P))[[K_P]^{-1}: P\in\P].$$

For each $\alpha\in K(\A)$, by writing $\alpha=\hat{P}-\hat{Q}$ for some objects $P,Q\in\P$, we define
$$K_{\alpha}:=[K_P]\ast[K_Q]^{-1}.$$ Then for any $\alpha,\beta\in K(\A)$ and $M_\bullet\in C_2(\P)$, \begin{flalign}&K_\alpha\ast K_\beta=K_{\alpha+\beta}=K_\beta\ast K_\alpha;\label{KK}\\
&K_\alpha\ast[M_\bullet]=[M_\bullet]\ast K_\alpha.\label{KX}\end{flalign}
For each $M\in\A$, we define
$$\X_M:=K_{-\hat{P}_M}\ast[C_M]\in\M\H(\A).$$
By \cite{DXZ}, $\X_M$ does not depend on the minimality of projective resolutions of $M$.
Given $M\in\A$ and $P\in\P$, we define
\begin{flalign*}\X_{M\oplus P[1]}:=\X_M\ast[Z_P]=K_{-\hat{P}_M}\ast[C_M]\ast[Z_P]=K_{-\hat{P}_M}\ast[C_M\oplus Z_P].
\end{flalign*}
In particular, \begin{equation}\label{XMP}\X_{P[1]}=[Z_P]\quad\text{and}\quad\X_{M\oplus P[1]}=\X_M\ast \X_{P[1]}.\end{equation}

Given objects $B, M, P,Q$ in $\A$, define
$${}_Q\Hom_{\A}(P,M)_B:=\{f:P\to M~|~\Ker(f)\cong Q~~\text{and}~~\Coker(f)\cong B\}.$$ By \cite{Vanden,ZHC},
\begin{equation}\label{xjs}|{}_Q\Hom_{\A}(P,M)_B|=\sum\limits_{[L]}a_LF^P_{LQ}F^M_{BL}.\end{equation}
\begin{theorem}\cite[Theorem 5.2]{DXZ}\label{ydygx}
The algebra $\mathcal{MH}(\A)$ is spanned by all $K_{\alpha}\ast\mathbb{X}_{M\oplus P[1]}$ $($with $\alpha\in K(\A)$, $M\in\A$ and $P\in\P$$)$,
which are subject to the following relations
\begin{flalign}
&K_{\alpha}\ast K_{\beta}=K_{\alpha+\beta}=K_{\beta}\ast K_{\alpha}\label{h1};\\
&K_{\alpha}\ast\mathbb{X}_{M\oplus P[1]}
    =\mathbb{X}_{M\oplus P[1]}\ast K_{\alpha}\label{h2};\\
&\mathbb{X}_{P[1]}\ast \mathbb{X}_{Q[1]}=
\mathbb{X}_{(P\oplus Q)[1]}
=\mathbb{X}_{Q[1]}\ast\mathbb{X}_{P[1]}\label{h3};\\
&\mathbb{X}_{M}\ast\mathbb{X}_{N}=q^{\lr{M,N}}\sum_{[L]}\frac{|\mathrm{Ext}_{\A}^{1}(M,N)_{L}|}{|\mathrm{Hom}_{\A}(M,N)|}\mathbb{X}_L\label{h4};\\
&\mathbb{X}_{M}\ast\mathbb{X}_{P[1]}=\mathbb{X}_{M\oplus
P[1]}\label{h5};\\
&\mathbb{X}_{P[1]}\ast\mathbb{X}_{M}=q^{-\lr{P,M}}
\sum\limits_{[B],[Q]}|{}_Q\Hom_{\A}(P,M)_B|\mathbb{X}_{B\oplus Q[1]}\label{h6};\end{flalign}
for any $\alpha,\beta\in K(\A)$, $M,N\in\A$ and $P,Q\in\P$.
\end{theorem}

\section{Bialgebra structure on the Hall algebra $\M\H(\A)$ of morphisms}
In this section, we define a structure of a bialgebra without counit on the Hall algebra $\M\H(\A)$ of morphisms. That is, we define a comultiplication map $\Delta$ on $\M\H(\A)$ such that $\Delta$ is an algebra homomorphism. Let us define the multiplication $\ast$ on $\M\H(\A)\otimes\M\H(\A)$ by
\begin{equation}\label{nihao}
\begin{split}
&(K_{\alpha}\ast\mathbb{X}_{M\oplus P[1]}\otimes K_{\beta}\ast\mathbb{X}_{N\oplus Q[1]})\ast(K_{\gamma}\ast\mathbb{X}_{U\oplus S[1]}\otimes K_{\delta}\ast\mathbb{X}_{V\oplus T[1]}):=\\
&q^{(\hat{N}-\hat{Q},\hat{U}-\hat{S})+\lr{\hat{M}-\hat{P},\hat{V}-\hat{T}}}
(K_{\alpha+\gamma}\ast\mathbb{X}_{M\oplus P[1]}\ast\mathbb{X}_{U\oplus S[1]}\otimes K_{\beta+\delta}\ast\mathbb{X}_{N\oplus Q[1]}\ast\mathbb{X}_{V\oplus T[1]})
\end{split}
\end{equation}
for any $\alpha,\beta,\gamma,\delta\in K(\A)$, $M,N,U,V\in\A$ and $P,Q,S,T\in\P$.
We also define a homomorphism of $\ZZ[v,v^{-1}]$-modules
\begin{equation*}
\Delta:\M\H(\A)\longrightarrow \M\H(\A)\otimes\M\H(\A)
\end{equation*}
by
\begin{equation}\label{yucheng}\Delta(K_{\alpha}\ast\mathbb{X}_{L\oplus P[1]}):=\sum\limits_{[M],[N]}
q^{\lr{\hat{M},\hat{N}-\hat{P}}}F_{MN}^L(\mathbb{X}_M\otimes K_\alpha\ast\mathbb{X}_{N\oplus P[1]})
\end{equation} for any $\alpha\in K(\A)$, $L\in\A$ and $P\in\P$.
\begin{remark}
Note that $\M\H(\A)$ is a $K(\A)$-graded algebra by defining the degree of $K_{\alpha}\ast\mathbb{X}_{M\oplus P[1]}$ to be $\hat{M}-\hat{P}$ for any $\alpha\in K(\A), M\in\A$ and $P\in\P$. For any elements $x_{ij}\in\M\H(\A)$ with the degree equal to $\alpha_{ij}\in K(\A)$, where $i=1,2,3$ and $j=1,2$, it is easy to see that $(\alpha_{12},\alpha_{21})+\lr{\alpha_{11},\alpha_{22}}+(\alpha_{12}+\alpha_{22},\alpha_{31})+\lr{\alpha_{11}+\alpha_{21},\alpha_{32}}
=(\alpha_{22},\alpha_{31})+\lr{\alpha_{21},\alpha_{32}}+(\alpha_{12},\alpha_{21}+\alpha_{31})+\lr{\alpha_{11},\alpha_{22}+\alpha_{32}}$, and then $[(x_{11}\otimes x_{12})\ast(x_{21}\otimes x_{22})]\ast(x_{31}\otimes x_{32})=(x_{11}\otimes x_{12})\ast[(x_{21}\otimes x_{22})\ast(x_{31}\otimes x_{32})]$, i.e., the twisted multiplication (\ref{nihao}) on $\M\H(\A)\otimes\M\H(\A)$ is also associative.
\end{remark}

\begin{proposition}\label{sdsjg}
The map $\Delta:(\M\H(\A),\ast)\longrightarrow (\M\H(\A)\otimes\M\H(\A),\ast)$ is a homomorphism of algebras.
\end{proposition}
\begin{proof}
It suffices to prove that $\Delta$ preserves all the relations in Theorem \ref{ydygx}. We only prove the relations
(\ref{h4}) and (\ref{h6}), since the other relations can be easily proved.

By Theorem \ref{ydygx} (see also \cite[Theorem 3.5]{DXZ}), the map $$\xymatrix{\Psi:\H_{\tw}(\A)\ar@{^{(}->}[r]&\M\H(\A),} \xymatrix{[M]\ar@{|->}[r]&\X_M}$$ is an embedding of algebras.
It follows from Green's formula (cf. \cite{Gr95}) that the map
$$\delta:\H_{\tw}(\A)\longrightarrow\H_{\tw}(\A)\otimes\H_{\tw}(\A)$$ defined by
\begin{equation}\delta([L])=\sum\limits_{[M],[N]}q^{\lr{M,N}}F_{MN}^L[M]\otimes[N]\end{equation} is a homomorphism of algebras, when the multiplication on $\H_{\tw}(\A)\otimes\H_{\tw}(\A)$ is given by
\begin{equation}\label{tttt}
([M_1]\otimes[N_1])\ast([M_2]\otimes[N_2]):=q^{(\hat{M}_2,\hat{N}_1)+\lr{\hat{M}_1,\hat{N}_2}}([M_1]\ast[M_2]\otimes[N_1]\ast[N_2]).\end{equation}
Here we remark that the twisted multiplication (\ref{tttt}) is also associative by considering the restriction of the twisted multiplication (\ref{nihao}) to $\H_{\tw}(\A)\otimes\H_{\tw}(\A)$.
By the commutative diagram
\begin{equation*}
\xymatrix{\H_{\tw}(\A)\ar[r]^-{\delta}\ar[d]_-{\Psi}& \H_{\tw}(\A)\otimes\H_{\tw}(\A)\ar[d]^-{\Psi\otimes\Psi}\\
\M\H(\A)\ar[r]^-{\Delta}&\M\H(\A)\otimes\M\H(\A),}
\end{equation*}
we obtain that $\Delta$ preserves the relation (\ref{h4}).

Next we prove that $\Delta$ preserves the relation (\ref{h6}). On the one hand,
\begin{flalign*}
\Delta(\mathbb{X}_{P[1]})\ast\Delta(\mathbb{X}_{L})&=\sum\limits_{[M],[N]}q^{\lr{\hat{M},\hat{N}}-(\hat{P},\hat{M})}F_{MN}^L\mathbb{X}_M\otimes\mathbb{X}_{P[1]}\ast\mathbb{X}_N\\
&=\sum\limits_{[M],[N],[Q],[C]}q^{\lr{\hat{M},\hat{N}}-(\hat{P},\hat{M})-\lr{\hat{P},\hat{N}}}F_{MN}^L|{}_Q\Hom_{\A}(P,N)_C|\mathbb{X}_M\otimes\mathbb{X}_{C\oplus Q[1]}\\
&=\sum\limits_{[M],[N],[Q],[C],[D]}q^{\lr{\hat{M},\hat{N}-\hat{P}}-\lr{\hat{P},\hat{M}+\hat{N}}}a_DF_{MN}^LF_{DQ}^PF_{CD}^N\mathbb{X}_M\otimes\mathbb{X}_{C\oplus Q[1]}\\
&=\sum\limits_{[M],[Q],[C],[D]}q^{\lr{\hat{M},\hat{C}-\hat{Q}}-\lr{\hat{P},\hat{L}}}a_DF_{MCD}^LF_{DQ}^P\mathbb{X}_M\otimes\mathbb{X}_{C\oplus Q[1]}.
\end{flalign*}
On the other hand,
\begin{flalign*}
&\sum\limits_{[B],[Q]}q^{-\lr{P,L}}|{}_Q\Hom_{\A}(P,L)_B|\Delta(\mathbb{X}_{B\oplus Q[1]})\\&=\sum\limits_{[B],[Q],[M],[C]}q^{\lr{\hat{M},\hat{C}-\hat{Q}}-\lr{\hat{P},\hat{L}}}|{}_Q\Hom_{\A}(P,L)_B|F_{MC}^B\mathbb{X}_M\otimes\mathbb{X}_{C\oplus Q[1]}\\
&=\sum\limits_{[B],[Q],[M],[C],[D]}q^{\lr{\hat{M},\hat{C}-\hat{Q}}-\lr{\hat{P},\hat{L}}}a_DF^P_{DQ}F^L_{BD}F_{MC}^B\mathbb{X}_M\otimes\mathbb{X}_{C\oplus Q[1]}\\
&=\sum\limits_{[Q],[M],[C],[D]}q^{\lr{\hat{M},\hat{C}-\hat{Q}}-\lr{\hat{P},\hat{L}}}a_DF^L_{MCD}F^P_{DQ}\mathbb{X}_M\otimes\mathbb{X}_{C\oplus Q[1]}.
\end{flalign*}
Thus, $$\Delta(\mathbb{X}_{P[1]})\ast\Delta(\mathbb{X}_{L})=\sum\limits_{[B],[Q]}q^{-\lr{P,L}}|{}_Q\Hom_{\A}(P,L)_B|\Delta(\mathbb{X}_{B\oplus Q[1]}).$$
\end{proof}
\begin{remark} $(1)$~Let $m: (\M\H(\A)\otimes\M\H(\A),\ast)\longrightarrow (\M\H(\A),\ast)$ be the multiplication map.
The Hall algebra $(\M\H(\A),m,\Delta)$ is a bialgebra without counit. In fact, for any $\mathbb{Z}[v,v^{-1}]$-module homomorphism $\epsilon:\M\H(\A)\rightarrow\mathbb{Z}[v,v^{-1}]$, $m\circ(\epsilon\otimes\id)\circ\Delta(\X_{P[1]})=\epsilon(1)\X_{P[1]}$, while $m\circ(\id\otimes\epsilon)\circ\Delta(\X_{P[1]})=\epsilon(\X_{P[1]})1$. Clearly, $m\circ(\epsilon\otimes\id)\circ\Delta(\X_{P[1]})\neq m\circ(\id\otimes\epsilon)\circ\Delta(\X_{P[1]})$.

$(2)$~The comultiplication defined in (\ref{yucheng}) is not the intrinsic comultiplication (cf. \cite{Gr95,Hubery}) of the Hall algebra of an exact category.
\end{remark}

\section{Integration map on the Hall algebra $\M\H(\A)$ of morphisms}
In this section, we define an integration map on the Hall algebra $\M\H(\A)$ of morphisms.
First of all, we give a characterization on the Grothendieck group $K(C_2(\P))$ of the morphism category $C_2(\P)$.
Assume that the rank of $K(\A)$ is equal to $n$. Using the basis consisting of simple $A$-modules, we identify $K(\A)$ with $\mathbb{Z}^n$ by sending $\hat{M}$ to $\Dim M$ for any $M\in\A$.

Given an object $M_\bullet=(\xymatrix{M_{-1}\ar[r]^-{f}&M_0})\in C_2(\P)$, we define
$$\Dim M_\bullet:={\Dim M_0-\Dim M_{-1}\choose\Dim M_0}\in\mathbb{Z}^{2n}.$$
For any short exact sequence $$0\longrightarrow M_\bullet \longrightarrow L_\bullet\longrightarrow N_\bullet\longrightarrow0$$ in $C_2(\P)$, it is easy to see that $\Dim L_\bullet=\Dim M_\bullet+\Dim N_\bullet$.

Let $P_1,P_2,\cdots,P_n$ be all indecomposable projective objects in $\A$ up to isomorphism.
\begin{lemma}\label{ftg}
The Grothendieck group $K(C_2(\P))$ is a free abelian group having as a basis the set
$$\{\hat{Z}_{P_i},\hat{K}_{P_i}~|~1\leq i\leq n\}$$
and there exists a unique group isomorphism $f: K(C_2(\P))\longrightarrow \mathbb{Z}^{2n}$ such that $f(\hat{M_\bullet})=\Dim M_\bullet$ for each object $M_\bullet$ in $C_2(\P)$.
\end{lemma}
\begin{proof}
For any object $M_\bullet\in C_2(\P)$ with $M_0=\oplus_{i=1}^na_iP_i$ and $M_{-1}=\oplus_{i=1}^nb_iP_i$, where $a_i, b_i\in\mathbb{N}$, $1\leq i\leq n$, by Lemma \ref{resol}, taking an injective resolution
$$0\longrightarrow M_{\bullet}\longrightarrow \oplus_{i=1}^nb_iZ_{P_i}\oplus \oplus_{i=1}^na_iK_{P_i}\longrightarrow \oplus_{i=1}^na_iZ_{P_i}\longrightarrow0,$$
we obtain that in $K(C_2(\P))$
$$\hat{M}_\bullet=\sum\limits_{i=1}^n((b_i-a_i)\hat{Z}_{P_i}+a_i\hat{K}_{P_i}).$$ This shows that $\{\hat{Z}_{P_i},\hat{ K}_{P_i}~|~1\leq i\leq n\}$ generates the group $K(C_2(\P))$.

For any objects $M_\bullet,N_\bullet\in C_2(\P)$, it is clear that $M_\bullet\cong N_\bullet$ implies $\Dim M_\bullet=\Dim N_\bullet$. Thus, the additivity of ${\bf dim}$ implies the existence of a unique group homomorphism $f: K(C_2(\P))\longrightarrow \mathbb{Z}^{2n}$ such that $f(\hat{M_\bullet})=\Dim M_\bullet$ for each object $M_\bullet$ in $C_2(\P)$. Since  $$\{\Dim Z_{P_i},\Dim K_{P_i}~|~1\leq i\leq n\}=\{{-{\bf dim\,} P_i\choose\mathbf{0}},{\mathbf{0}\choose{\bf dim\,}P_i}~|~1\leq i\leq n\}$$ is a basis of the free group $\mathbb{Z}^{2n}$, we conclude that $\{\hat{Z}_{P_i},\hat{ K}_{P_i}~|~1\leq i\leq n\}$ is $\mathbb{Z}$-linearly independent in $K(C_2(\P))$. It follows that $K(C_2(\P))$ is free and $f$ is an isomorphism.
\end{proof}

Let $\mathcal{T}$ be the $\ZZ[v,v^{-1}]$-algebra with a basis $\{X^{\alpha}~|~\alpha\in \mathbb{Z}^{2n}\}$ and
multiplication defined by
\begin{equation}\label{jhtorus}X^{\alpha}\diamond X^{\beta}=X^{\alpha+\beta}.\end{equation}
We give integration maps on Hall algebras of morphisms in the following
\begin{proposition}\label{jfys1}
The integration map $$\int:\H_{\tw}(C_2(\P))\longrightarrow\mathcal{T},~~[M_\bullet]\mapsto X^{{\bf dim\,}{M_\bullet}}$$
is a homomorphism of algebras.
\end{proposition}
\begin{proof}
For any objects $M_\bullet,N_\bullet\in C_2(\P)$,
\begin{flalign*}\int[M_\bullet]\ast[N_\bullet] &= \sum\limits_{[L_\bullet]} q^{\lr{M_\bullet,N_\bullet}}{\frac{{|\Ext_{C_2(\P)}^1{{(M_\bullet,N_\bullet)}_{L_\bullet}}|}}{{|\Hom_{C_2(\P)}(M_\bullet,N_\bullet)|}}} X^{{\bf dim\,}{L_\bullet}}\\
&=\sum\limits_{[L_\bullet]} q^{\lr{M_\bullet,N_\bullet}}{\frac{{|\Ext_{C_2(\P)}^1{{(M_\bullet,N_\bullet)}_{L_\bullet}}|}}{{|\Hom_{C_2(\P)}(M_\bullet,N_\bullet)|}}} X^{{\bf dim\,}{M_\bullet}+{\bf dim\,}{N_\bullet}}\\
&=X^{{\bf dim\,}{M_\bullet}+{\bf dim\,}{N_\bullet}}\\
&=X^{{\bf dim\,}{M_\bullet}}\diamond X^{{\bf dim\,}{N_\bullet}}\\
&=\int[M]\diamond\int[N].\end{flalign*}
\end{proof}
In what follows,
for each $\alpha\in \mathbb{Z}^{n}$, we always set $\widetilde{\alpha}:={\mathbf{0}\choose\alpha}
\in \mathbb{Z}^{2n}$ and $\bar{\alpha}:={\alpha\choose\mathbf{0}}
\in \mathbb{Z}^{2n}$. For each given object in $\A$ we will always use the
corresponding lowercase boldface letter to denote its dimension vector.
\begin{corollary}
The integration map $$\int:\M\H(\A)\longrightarrow\mathcal{T},~~K_\alpha\ast[M_\bullet]\mapsto X^{\tilde{\alpha}+{\bf dim\,}{M_\bullet}}$$
is a homomorphism of algebras. In particular, $\int K_\alpha=X^{\tilde{\alpha}}$ and $\int \X_{M\oplus P[1]}=X^{\bar{\mathbf{m}}-\bar{\mathbf{p}}}$.
\end{corollary}
\begin{proof}
By definition, $\int K_\alpha=X^{\tilde{\alpha}}$ and $$\int (K_\alpha\ast[M_\bullet])=X^{\tilde{\alpha}+{\bf dim\,}{M_\bullet}}=X^{\tilde{\alpha}}\diamond X^{{\bf dim\,}{M_\bullet}}=\int K_\alpha\diamond\int [M_\bullet].$$
By Proposition \ref{jfys1} and the relation (\ref{KX}), we can prove that $\int$ is a homomorphism of algebras. Thus,
\begin{flalign*}
\int \X_{M\oplus P[1]}&=\int (K_{-\hat{P}_M}\ast[C_M]\ast[Z_P])\\
&=\int K_{-\hat{P}_M}\diamond\int[C_M]\diamond\int[Z_P]\\
&=X^{\mathbf{0}\choose-{\bf dim\,} P_M}\diamond X^{{\bf dim\,} M\choose{\bf dim\,} P_M}\diamond X^{{-{\bf dim\,} P\choose\mathbf{0}}}\\
&=X^{{\bf dim\,} M-{\bf dim\,} P\choose\mathbf{0}}\\
&=X^{\bar{\mathbf{m}}-\bar{\mathbf{p}}}.
\end{flalign*}
Therefore, we complete the proof.
\end{proof}

\section{Compatible pairs}
Let $Q$ be an acyclic valued quiver (cf. \cite{Rupel1,Rupel2}) with the vertex set $\{1,2,\cdots,n\}$. For each vertex $i$, let $d_i\in\mathbb{N^+}$ be the corresponding valuation. Note that each finite dimensional hereditary $k$-algebra can be obtained by taking the tensor algebra of the $k$-species $\mathfrak{S}$ associated to $Q$. We identify a $k$-species $\mathfrak{S}$ with its corresponding tensor algebra. Let $m\geq n$, we define a new quiver $\widetilde{Q}$ by attaching additional vertices $n+1,\ldots,m$ to $Q$. The full subquiver $Q$ is called the principal part of $\widetilde{Q}$.

For $1\leq i\leq m$, denote by $S_i$ the $i$-th simple module for $\widetilde{\mathfrak{S}}$ which is the $k$-species associated to $\widetilde{Q}$, and set $\mathcal {D}_i=\End_{\widetilde{\mathfrak{S}}}(S_i)$.
Let $\widetilde{R}$ and $\widetilde{R}'$ be the $m\times n$ matrices with the $i$-th row and $j$-th column elements given respectively by
$$r_{ij}=\dim_{\mathcal {D}_i}\Ext_{\widetilde{\mathfrak{S}}}^1(S_j,S_i)$$
and
$$r'_{ij}=\dim_{{\mathcal {D}_i}^{op}}\Ext_{\widetilde{\mathfrak{S}}}^1(S_i,S_j),$$
where $1\leq i\leq m$ and $1\leq j\leq n$. Define $\widetilde{B}=\widetilde{R}'-\widetilde{R}$, and denote by $\widetilde{I}$ the left $m\times n$ submatrix of the $m\times m$ identity matrix $I_m$. Denote the principal parts of the matrices $\widetilde{R}$, $\widetilde{R}'$ and $\widetilde{B}$ by $R$, $R'$ and $B$, respectively. That is, $B=R'-R$. Let $D_n=diag(d_1,\cdots,d_n)$,
it is easy to see that $D_nR'=R^{\rm tr}D_n$. Thus, $D_nB$ is skew-symmetric.
Moreover, the matrix representing the Euler form associated to $\mathfrak{S}$ under the standard basis is $(I_n-R^{\rm tr})D_n=D_n(I_n-R')$.

We always assume that there exists a skew-symmetric $m\times m$ integral
matrix $\Lambda$ such that
\begin{align}\label{compatible}
\Lambda(-\widetilde{B})={D_n\choose0}.\end{align} We remark that such $\widetilde{Q}$ and $\Lambda$ exist for a given quiver $Q$ (cf. \cite{Rupel1}). We call such $(\Lambda,\widetilde{B})$ a {\em compatible pair}. In what follows, we also denote by $\Lambda$ the skew-symmetric bilinear form on $\mathbb{Z}^{m}$ associated to the skew-symmetric matrix $\Lambda$.

For simplicity of notation, we set $\widetilde{E}=\widetilde{I}-\widetilde{R}'$ and $\widetilde{E}'=\widetilde{I}-\widetilde{R}$. Thus, $\widetilde{E}'-\widetilde{E}=\widetilde{B}$.

\begin{lemma}\cite[Lemma 3.1]{DSC}\label{eulerjs}
For any $\alpha,\beta\in\mathbb{Z}^n$, we have that
\vspace{0.8em}

$(1)$~~$\Lambda(\widetilde{E}\alpha,\widetilde{B}\beta)=-\lr{\beta,\alpha}$;\quad
$(2)$~~$\Lambda(\widetilde{B}\alpha,\widetilde{B}\beta)=\lr{\beta,\alpha}-\lr{\alpha,\beta}$.
\end{lemma}

Using Lemma \ref{eulerjs}, we easily obtain the following
\begin{lemma}\label{xq0}
\begin{itemize}
\item[(1)]For any $\alpha,\beta\in\mathbb{Z}^n$, we have that $$\Lambda(\widetilde{E}'\alpha,\widetilde{E}'\beta)=\Lambda(\widetilde{E}\alpha,\widetilde{E}\beta).$$
\item[(2)]For any $\alpha_i,\beta_i\in\mathbb{Z}^n$, $i=1,2$, we have that
$$\Lambda(\widetilde{E}\alpha_1+\widetilde{E}'\beta_1,\widetilde{E}\alpha_2+\widetilde{E}'\beta_2)
=\Lambda(\widetilde{E}(\alpha_1+\beta_1),\widetilde{E}(\alpha_2+\beta_2))-\lr{\beta_2,\alpha_1}+\lr{\beta_1,\alpha_2}.$$
\end{itemize}
\end{lemma}
\begin{proof}$(1)$ By definition,
\begin{flalign*}
\Lambda(\widetilde{E}'\alpha,\widetilde{E}'\beta)&=\Lambda(\widetilde{B}\alpha+\widetilde{E}\alpha,\widetilde{B}\beta+\widetilde{E}\beta)\\
&=\Lambda(\widetilde{B}\alpha,\widetilde{B}\beta)+\Lambda(\widetilde{B}\alpha,\widetilde{E}\beta)+\Lambda(\widetilde{E}\alpha,\widetilde{B}\beta)+\Lambda(\widetilde{E}\alpha,\widetilde{E}\beta)\\
&=\lr{\beta,\alpha}-\lr{\alpha,\beta}+\lr{\alpha,\beta}-\lr{\beta,\alpha}+\Lambda(\widetilde{E}\alpha,\widetilde{E}\beta)\\
&=\Lambda(\widetilde{E}\alpha,\widetilde{E}\beta).
\end{flalign*}
$(2)$ By definition,
\begin{flalign*}
\Lambda(\widetilde{E}\alpha_1+\widetilde{E}'\beta_1,\widetilde{E}\alpha_2+\widetilde{E}'\beta_2)&=
\Lambda(\widetilde{E}(\alpha_1+\beta_1)+\widetilde{B}\beta_1,\widetilde{E}(\alpha_2+\beta_2)+\widetilde{B}\beta_2)\\&=
\Lambda(\widetilde{E}(\alpha_1+\beta_1),\widetilde{E}(\alpha_2+\beta_2))+\Lambda(\widetilde{E}(\alpha_1+\beta_1),\widetilde{B}\beta_2)\\&\quad+
\Lambda(\widetilde{B}\beta_1,\widetilde{E}(\alpha_2+\beta_2))+\Lambda(\widetilde{B}\beta_1,\widetilde{B}\beta_2)\\
&=\Lambda(\widetilde{E}(\alpha_1+\beta_1),\widetilde{E}(\alpha_2+\beta_2))-\lr{\beta_2,\alpha_1+\beta_1}+\lr{\beta_1,\alpha_2+\beta_2}\\&\quad+\lr{\beta_2,\beta_1}-\lr{\beta_1,\beta_2}\\
&=\Lambda(\widetilde{E}(\alpha_1+\beta_1),\widetilde{E}(\alpha_2+\beta_2))-\lr{\beta_2,\alpha_1}+\lr{\beta_1,\alpha_2}.
\end{flalign*}
\end{proof}

\section{Factorization of a homomorphism}
Let $Q$, $\widetilde{Q}$ be the same as given in Section 5. We take $\A$ to be the category of finite dimensional left $\mathfrak{S}$-modules, and $m=2n$. Recall that for each $\alpha\in \mathbb{Z}^{n}$, $\widetilde{\alpha}:={\mathbf{0}\choose\alpha}
\in \mathbb{Z}^{m}$.

By \cite[Lemma 8.1]{DXZ}, the algebra $\mathcal{MH}(\A)$ is $\mathbb{Z}^{m}$-graded with the degree of $K_{\alpha}\ast\mathbb{X}_{M\oplus P[1]}$ defined by $$\deg(K_{\alpha}\ast\mathbb{X}_{M\oplus P[1]}):=\widetilde{E}'(\bf{m}-\bf{p})-\widetilde{\alpha}$$ for any $\alpha\in \mathbb{Z}^n, M\in\A$ and $P\in\P$. It induces a grading on the tensor algebra $(\M\H(\A)\otimes\M\H(\A),\ast)$ defined by
$$\deg(K_{\alpha}\ast\mathbb{X}_{M\oplus P[1]}\otimes K_{\beta}\ast\mathbb{X}_{N\oplus Q[1]}):=\widetilde{E}'(\bf{m}-\bf{p}+\bf{n}-\bf{q})-\widetilde{\alpha}-\widetilde{\beta}$$ for any $\alpha, \beta\in \mathbb{Z}^n$, $M, N\in\A$ and $P, Q\in\P$.
Using these gradings, we twist the multiplication on $\M\H(\A)$, and define $\mathcal{MH}_{\Lambda}(\A)$ to be the same module as $\mathcal{MH}(\A)$ but with the twisted
multiplication defined on basis elements by
\begin{equation}\label{labuda}
\begin{split}
   &(K_{\alpha}\ast\mathbb{X}_{M\oplus P[1]})\star (K_{\beta}\ast\mathbb{X}_{N\oplus Q[1]}):=\\&
   v^{\Lambda(\widetilde{E}'({\bf m}-{\bf p})-\widetilde{\alpha},\widetilde{E}'(\bf{n}
   -\bf{q})-\widetilde{\beta})}
   (K_{\alpha}\ast\mathbb{X}_{M\oplus P[1]})\ast (K_{\beta}\ast\mathbb{X}_{N\oplus Q[1]}),\end{split}
\end{equation}
where $\alpha,\beta\in \mathbb{Z}^{n}$, $M, N\in\A$ and $P, Q\in\P$. We also twist the multiplication on the tensor algebra $(\M\H(\A)\otimes\M\H(\A),\ast)$ by defining
\begin{equation}\begin{split}&(K_{\alpha}\ast\mathbb{X}_{M\oplus P[1]}\otimes K_{\beta}\ast\mathbb{X}_{N\oplus Q[1]})\star(K_{\gamma}\ast\mathbb{X}_{U\oplus S[1]}\otimes K_{\delta}\ast\mathbb{X}_{V\oplus T[1]}):=\\&
v^{\lambda}
(K_{\alpha}\ast\mathbb{X}_{M\oplus P[1]}\otimes K_{\beta}\ast\mathbb{X}_{N\oplus Q[1]})\ast(K_{\gamma}\ast\mathbb{X}_{U\oplus S[1]}\otimes K_{\delta}\ast\mathbb{X}_{V\oplus T[1]}),
\end{split}
\end{equation}where $\lambda=\Lambda(\widetilde{E}'({\bf m}-{\bf p}+{\bf n}-{\bf q})-\widetilde{\alpha}-\widetilde{\beta},\widetilde{E}'(\bf{u}-\bf{s}+\bf{v}-\bf{t})-\widetilde{\gamma}-\widetilde{\delta})$, $M, N,U,V\in\A$, $P, Q,S,T\in\P$ and $\alpha,\beta,\gamma,\delta\in \mathbb{Z}^{n}$.

For use below we reformulate Theorem \ref{ydygx} in the following
\begin{proposition}\label{tdygx}
The algebra $\mathcal{MH}_{\Lambda}(\A)$ is generated by all $K_{\alpha}$ and $\mathbb{X}_{M\oplus P[1]}$ $($with $\alpha\in\mathbb{Z}^{n}$, $M\in\A$ and $P\in\P$$)$,
which are subject to the following relations
\begin{flalign}&K_{\alpha}\star K_{\beta}=v^{\Lambda(\widetilde{\alpha},\widetilde{\beta})}K_{\alpha+\beta}=q^{\Lambda(\widetilde{\alpha},\widetilde{\beta})}K_{\beta}\star K_{\alpha}\label{x1};\\
&K_{\alpha}\star\mathbb{X}_{M\oplus P[1]}
    =q^{-\Lambda(\widetilde{\alpha},\widetilde{E}'({\bf m}-{\bf p}))}\mathbb{X}_{M\oplus P[1]}\star K_{\alpha}\label{x2};\\
&\mathbb{X}_{P[1]}\star \mathbb{X}_{Q[1]}=
v^{\Lambda(\widetilde{E}'{\bf p},\widetilde{E}'\bf{q}
)}\mathbb{X}_{(P\oplus Q)[1]}
=q^{\Lambda(\widetilde{E}'{\bf p},\widetilde{E}'\bf{q})}\mathbb{X}_{Q[1]}\star\mathbb{X}_{P[1]}\label{x3};\\
&\mathbb{X}_{M}\star\mathbb{X}_{N}=q^{\frac{1}{2}\Lambda(\widetilde{E}'{\bf m},
\widetilde{E}'\bf{n})+\langle\bf{m},\bf{n}\rangle}\sum_{[L]}\frac{|\mathrm{Ext}_{\A}^{1}(M,N)_{L}|}{|\mathrm{Hom}_{\A}(M,N)|}\mathbb{X}_L\label{x4};\\
&\mathbb{X}_{M}\star\mathbb{X}_{P[1]}=v^{-\Lambda(\widetilde{E}'{\bf m},\widetilde{E}'{\bf p})}\mathbb{X}_{M\oplus
P[1]}\label{x5};\\
&\mathbb{X}_{P[1]}\star\mathbb{X}_{M}=q^{-\frac{1}{2}\Lambda(\widetilde{E}'{\bf p},\widetilde{E}'{\bf m})-\lr{{\bf p},{\bf m}}}
\sum\limits_{[B],[Q]}|{}_Q\Hom_{\A}(P,M)_B|\mathbb{X}_{B\oplus Q[1]}\label{x6};\end{flalign}
for any $\alpha,\beta\in\mathbb{Z}^n$, $M,N\in\A$ and $P,Q\in\P$.
\end{proposition}

Since the comultiplication $\Delta$ defined in (\ref{yucheng}) is homogeneous, it is easy to obtain the following
\begin{lemma}\label{ts1}
The map $\Delta:(\M\H_\Lambda(\A),\star)\longrightarrow (\M\H(\A)\otimes\M\H(\A),\star)$ is a homomorphism of algebras.
\end{lemma}
\begin{proof}
Let $M, N\in\A$, $P, Q\in\P$ and $\alpha,\beta\in \mathbb{Z}^{n}$. By (\ref{labuda}), we have that
\begin{flalign*}&\Delta[(K_{\alpha}\ast\mathbb{X}_{M\oplus P[1]})\star (K_{\beta}\ast\mathbb{X}_{N\oplus Q[1]})]\\&=
v^{\Lambda(\widetilde{E}'({\bf m}-{\bf p})-\widetilde{\alpha},\widetilde{E}'(\bf{n}
   -\bf{q})-\widetilde{\beta})}
   \Delta[(K_{\alpha}\ast\mathbb{X}_{M\oplus P[1]})\ast (K_{\beta}\ast\mathbb{X}_{N\oplus Q[1]})]\\&=v^{\Lambda(\widetilde{E}'({\bf m}-{\bf p})-\widetilde{\alpha},\widetilde{E}'(\bf{n}
   -\bf{q})-\widetilde{\beta})}
   \Delta(K_{\alpha}\ast\mathbb{X}_{M\oplus P[1]})\ast \Delta(K_{\beta}\ast\mathbb{X}_{N\oplus Q[1]})\quad \text{(by~Proposition~\ref{sdsjg})}\\
   &=\Delta(K_{\alpha}\ast\mathbb{X}_{M\oplus P[1]})\star \Delta(K_{\beta}\ast\mathbb{X}_{N\oplus Q[1]}).
\end{flalign*}
\end{proof}

We twist the multiplication on the tensor algebra of the torus $\mathcal{T}$ defined in (\ref{jhtorus}) by defining
\begin{equation}\label{taozl}\begin{split}
&(X^{\alpha_1\choose\alpha_2}\otimes X^{\beta_1\choose\beta_2})\star(X^{\gamma_1\choose\gamma_2}\otimes X^{\delta_1\choose\delta_2}):=\\&
q^{\frac{1}{2}\Lambda(\widetilde{E}'(\alpha_1+\beta_1)-\widetilde{\alpha}_2-\widetilde{\beta}_2,\widetilde{E}'(\gamma_1+\delta_1)-\widetilde{\gamma}_2-\widetilde{\delta}_2)+(\beta_1,\gamma_1)+\lr{\alpha_1,\delta_1}}
X^{\alpha_1+\gamma_1\choose\alpha_2+\gamma_2}\otimes X^{\beta_1+\delta_1\choose\beta_2+\delta_2}
\end{split}
\end{equation}
for any $\alpha_i, \beta_i, \gamma_i, \delta_i\in\mathbb{Z}^n, i=1, 2.$
\begin{lemma}\label{ts2}
The map $\int\otimes\int: (\M\H(\A)\otimes\M\H(\A),\star)\longrightarrow(\mathcal{T}\otimes\mathcal{T},\star)$ is a homomorphism of algebras.
\end{lemma}
\begin{proof}
Let $M, N,U,V\in\A$, $P, Q,S,T\in\P$ and $\alpha,\beta,\gamma,\delta\in \mathbb{Z}^{n}$. On the one hand,
\begin{flalign*}
&\int\otimes\int[(K_{\alpha}\ast\mathbb{X}_{M\oplus P[1]}\otimes K_{\beta}\ast\mathbb{X}_{N\oplus Q[1]})\star(K_{\gamma}\ast\mathbb{X}_{U\oplus S[1]}\otimes K_{\delta}\ast\mathbb{X}_{V\oplus T[1]})]\\&
=v^{a_0}\int\otimes\int[(K_{\alpha}\ast\mathbb{X}_{M\oplus P[1]}\otimes K_{\beta}\ast\mathbb{X}_{N\oplus Q[1]})\ast(K_{\gamma}\ast\mathbb{X}_{U\oplus S[1]}\otimes K_{\delta}\ast\mathbb{X}_{V\oplus T[1]})]\\&
=v^{a_1}\int\otimes\int[(K_{\alpha+\gamma}\ast \mathbb{X}_{M\oplus P[1]}\ast\mathbb{X}_{U\oplus S[1]})\otimes
(K_{\beta+\delta}\ast \mathbb{X}_{N\oplus Q[1]}\ast\mathbb{X}_{V\oplus T[1]})]\\&
=v^{a_1}X^{\widetilde{\alpha}+\widetilde{\gamma}+\bar{\mathbf{m}}-\bar{\mathbf{p}}+\bar{\mathbf{u}}-\bar{\mathbf{s}}}\otimes
X^{\widetilde{\beta}+\widetilde{\delta}+\bar{\mathbf{n}}-\bar{\mathbf{q}}+\bar{\mathbf{v}}-\bar{\mathbf{p}}},
\end{flalign*}
where $a_0=\Lambda(\widetilde{E}'({\bf m}-{\bf p}+{\bf n}-{\bf q})-\widetilde{\alpha}-\widetilde{\beta},\widetilde{E}'(\bf{u}-\bf{s}+\bf{v}-\bf{t})-\widetilde{\gamma}-\widetilde{\delta})$
and $a_1=a_0+2({\bf n}-{\bf q},{\bf u}-{\bf s})+2\lr{\bf{m}-\bf{p},\bf{v}-\bf{t}}$.

On the other hand, \begin{flalign*}
&\int\otimes\int(K_{\alpha}\ast\mathbb{X}_{M\oplus P[1]}\otimes K_{\beta}\ast\mathbb{X}_{N\oplus Q[1]})\star\int\otimes\int(K_{\gamma}\ast\mathbb{X}_{U\oplus S[1]}\otimes K_{\delta}\ast\mathbb{X}_{V\oplus T[1]})\\&
=(X^{\widetilde{\alpha}+\bar{\bf{m}}-\bar{\bf{p}}}\otimes X^{\widetilde{\beta}+\bar{\bf{n}}-\bar{\bf{q}}})\star
(X^{\widetilde{\gamma}+\bar{\bf{u}}-\bar{\bf{s}}}\otimes X^{\widetilde{\delta}+\bar{\bf{v}}-\bar{\bf{t}}})\\&
=v^{a_1}X^{\widetilde{\alpha}+\widetilde{\gamma}+\bar{\mathbf{m}}-\bar{\mathbf{p}}+\bar{\mathbf{u}}-\bar{\mathbf{s}}}\otimes
X^{\widetilde{\beta}+\widetilde{\delta}+\bar{\mathbf{n}}-\bar{\mathbf{q}}+\bar{\mathbf{v}}-\bar{\mathbf{p}}}.\quad(\text{by~the~definition~in~(\ref{taozl}))}
\end{flalign*}
\end{proof}

Define the quantum torus $\mathcal{T}_\Lambda$ to be the $\ZZ[v,v^{-1}]$-algebra with a basis $\{X^{\alpha}~|~\alpha\in \mathbb{Z}^{m}\}$ and
multiplication defined by
\begin{equation}\label{ljhtorus}X^{\alpha}\star X^{\beta}=v^{\Lambda(\alpha,\beta)}X^{\alpha+\beta}.\end{equation}

For the need of the sequel proof, we give the following
\begin{lemma}\label{xq}
For any $\alpha,\beta\in\mathbb{Z}^n$, we have that
\vspace{0.8em}

$(1)$~~$\Lambda(\widetilde{E}'\alpha,\widetilde{\beta})=\Lambda(\widetilde{E}\alpha,\widetilde{\beta})$;\quad
$(2)$~~$\Lambda(\widetilde{\alpha},\widetilde{E}'\beta)=\Lambda(\widetilde{\alpha},\widetilde{E}\beta)$.
\end{lemma}
\begin{proof}
By (\ref{compatible}), we obtain that $(D_n~0)=\widetilde{B}^{\rm tr}\Lambda=(\widetilde{E}')^{\rm tr}\Lambda-(\widetilde{E})^{\rm tr}\Lambda$. Thus,
\begin{flalign*}
\Lambda(\widetilde{E}'\alpha,\widetilde{\beta})&=\alpha^{\rm tr}(\widetilde{E}')^{\rm tr}\Lambda\widetilde{\beta}\\&=\alpha^{\rm tr}(\widetilde{E})^{\rm tr}\Lambda\widetilde{\beta}+\alpha^{\rm tr}(D_n0)\widetilde{\beta}\\
&=\Lambda(\widetilde{E}\alpha,\widetilde{\beta}).
\end{flalign*}
Similarly, we can prove the second identity.
\end{proof}

\begin{proposition}\label{ts3}
The map $\mu: (\mathcal{T}\otimes\mathcal{T},\star)\longrightarrow(\mathcal{T}_\Lambda,\star)$ defined by
$$\mu(X^{\alpha_1\choose\alpha_2}\otimes X^{\beta_1\choose\beta_2})=v^{-(\alpha_1,\beta_1)-\lr{\alpha_1,\beta_1}}X^{-\widetilde{E}\alpha_1-\widetilde{E}'\beta_1+\widetilde{\alpha}_2+\widetilde{\beta}_2},$$
where $\alpha_i,\beta_i\in\mathbb{Z}^n$, $i=1,2$, is a homomorphism of algebras.
\end{proposition}
\begin{proof}
Let $\alpha_i,\beta_i,\gamma_i,\delta_i\in\mathbb{Z}^n$, $i=1,2$. On the one hand,
\begin{flalign*}\mu[(X^{\alpha_1\choose\alpha_2}\otimes X^{\beta_1\choose\beta_2})\star(X^{\gamma_1\choose\gamma_2}\otimes X^{\delta_1\choose\delta_2})]&=v^{a_0}X^{\alpha_1+\gamma_1\choose\alpha_2+\gamma_2}\otimes X^{\beta_1+\delta_1\choose\beta_2+\delta_2}\\
&=v^{a_1}X^{-\widetilde{E}(\alpha_1+\gamma_1)-\widetilde{E}'(\beta_1+\delta_1)+\widetilde{\alpha}_2+\widetilde{\gamma}_2+\widetilde{\beta}_2+\widetilde{\delta}_2},
\end{flalign*}
where $a_0=\Lambda(\widetilde{E}'(\alpha_1+\beta_1)-\widetilde{\alpha}_2-\widetilde{\beta}_2,\widetilde{E}'(\gamma_1+\delta_1)-\widetilde{\gamma}_2-\widetilde{\delta}_2)+2(\beta_1,\gamma_1)+2\lr{\alpha_1,\delta_1}$,
and $a_1=a_0-(\alpha_1+\gamma_1,\beta_1+\delta_1)-\lr{\alpha_1+\gamma_1,\beta_1+\delta_1}$.
By simple calculation, we obtain that
\begin{flalign*}a_1=&\Lambda(\widetilde{E}'(\alpha_1+\beta_1)-\widetilde{\alpha}_2-\widetilde{\beta}_2,\widetilde{E}'(\gamma_1+\delta_1)-\widetilde{\gamma}_2-\widetilde{\delta}_2)+
\lr{\beta_1,\gamma_1}-\lr{\delta_1,\alpha_1}\\&-(\alpha_1,\beta_1)-(\gamma_1,\delta_1)-\lr{\alpha_1,\beta_1}-\lr{\gamma_1,\delta_1}.\end{flalign*}

On the other hand,
\begin{flalign*}
\mu(X^{\alpha_1\choose\alpha_2}\otimes X^{\beta_1\choose\beta_2})\star\mu(X^{\gamma_1\choose\gamma_2}\otimes X^{\delta_1\choose\delta_2})&=v^{b_0}X^{-\widetilde{E}\alpha_1-\widetilde{E}'\beta_1+\widetilde{\alpha}_2+\widetilde{\beta}_2}\star
X^{-\widetilde{E}\gamma_1-\widetilde{E}'\delta_1+\widetilde{\gamma}_2+\widetilde{\delta}_2}\\&=v^{b_1}X^{-\widetilde{E}(\alpha_1+\gamma_1)-\widetilde{E}'(\beta_1+\delta_1)+\widetilde{\alpha}_2+\widetilde{\gamma}_2+\widetilde{\beta}_2+\widetilde{\delta}_2}
\end{flalign*}
where $b_0=-(\alpha_1,\beta_1)-\lr{\alpha_1,\beta_1}-(\gamma_1,\delta_1)-\lr{\gamma_1,\delta_1}$, and $b_1=b_0+\Lambda(-\widetilde{E}\alpha_1-\widetilde{E}'\beta_1+\widetilde{\alpha}_2+\widetilde{\beta}_2,-\widetilde{E}\gamma_1-\widetilde{E}'\delta_1+\widetilde{\gamma}_2+\widetilde{\delta}_2)$.

Since \begin{flalign*} &\Lambda(\widetilde{E}'(\alpha_1+\beta_1)-\widetilde{\alpha}_2-\widetilde{\beta}_2,\widetilde{E}'(\gamma_1+\delta_1)-\widetilde{\gamma}_2-\widetilde{\delta}_2)-\Lambda(-\widetilde{E}\alpha_1-\widetilde{E}'\beta_1+\widetilde{\alpha}_2+\widetilde{\beta}_2,-\widetilde{E}\gamma_1-\widetilde{E}'\delta_1+\widetilde{\gamma}_2+\widetilde{\delta}_2)\\
&=\Lambda(\widetilde{E}'(\alpha_1+\beta_1),\widetilde{E}'(\gamma_1+\delta_1))-\Lambda(\widetilde{E}\alpha_1+\widetilde{E}'\beta_1,\widetilde{E}\gamma_1+\widetilde{E}'\delta_1)
-\Lambda(\widetilde{E}'(\alpha_1+\beta_1),\widetilde{\gamma}_2+\widetilde{\delta}_2)\\
&+\Lambda(\widetilde{E}\alpha_1+\widetilde{E}'\beta_1,\widetilde{\gamma}_2+\widetilde{\delta}_2)-\Lambda(\widetilde{\alpha}_2+\widetilde{\beta}_2,\widetilde{E}'(\gamma_1+\delta_1))
+\Lambda(\widetilde{\alpha}_2+\widetilde{\beta}_2,\widetilde{E}\gamma_1+\widetilde{E}'\delta_1)\\
&=\Lambda(\widetilde{E}'(\alpha_1+\beta_1),\widetilde{E}'(\gamma_1+\delta_1))-\Lambda(\widetilde{E}\alpha_1+\widetilde{E}'\beta_1,\widetilde{E}\gamma_1+\widetilde{E}'\delta_1) \quad(\text{by~Lemma~\ref{xq}})\\
&=\lr{\delta_1,\alpha_1}-\lr{\beta_1,\gamma_1},\quad(\text{by~Lemma~\ref{xq0}})\end{flalign*}
we obtain that $a_1=b_1$. Therefore,
$$\mu[(X^{\alpha_1\choose\alpha_2}\otimes X^{\beta_1\choose\beta_2})\star(X^{\gamma_1\choose\gamma_2}\otimes X^{\delta_1\choose\delta_2})]=\mu(X^{\alpha_1\choose\alpha_2}\otimes X^{\beta_1\choose\beta_2})\star\mu(X^{\gamma_1\choose\gamma_2}\otimes X^{\delta_1\choose\delta_2}).$$
\end{proof}

Define $\Psi:=\mu\circ(\int\otimes\int)\circ\Delta$. In other words, we have the following commutative diagram
\begin{equation}
\xymatrix{(\M\H_\Lambda(\A),\star)\ar[r]^-{\Psi}\ar[d]_-{\Delta}&(\mathcal{T}_\Lambda,\star)\\(\M\H(\A)\otimes\M\H(\A),\star)\ar[r]^-{\int\otimes\int}&(\mathcal{T}\otimes\mathcal{T},\star).\ar[u]_-{\mu}}
\end{equation}
Given a module $M\in\A$, we denote by $\mathrm{Gr}_{\bf{e}}M$ the set of all submodules $V$
of $M$ with $\Dim V= \bf{e}$. Let us present our main result as the following
\begin{theorem}\label{main result}
The map $\Psi:(\M\H_\Lambda(\A),\star)\longrightarrow(\mathcal{T}_\Lambda,\star)$ is a homomorphism of algebras. Moreover, for any $\alpha\in\mathbb{Z}^n$, $M\in\A$ and $P\in\P$,
$$\Psi(K_\alpha)=X^{\widetilde{\alpha}}\quad\text{and}\quad\Psi(\X_{M\oplus P[1]})=\sum\limits_{\mathbf{e}}v^{\lr{\mathbf{p}-\mathbf{e},\mathbf{m}-\mathbf{e}}}|\mathrm{Gr}_{\mathbf{e}}M|
X^{\widetilde{E}'(\mathbf{p}-\mathbf{e})-\widetilde{E}(\mathbf{m}-\mathbf{e})}.$$
\end{theorem}
\begin{proof}
By Lemma \ref{ts1}, Lemma \ref{ts2} and Proposition \ref{ts3}, we obtain that $\Psi$ is a homomorphism of algebras.

For any $\alpha\in\mathbb{Z}^n$, \begin{flalign*}\Psi(K_\alpha)&=\mu\circ(\int\otimes\int)\circ\Delta(K_\alpha)\\&=\mu\circ(\int\otimes\int)(1\otimes K_\alpha)\\&=\mu(1\otimes X^{\widetilde{\alpha}})=X^{\widetilde{\alpha}}.\end{flalign*}

For any $M\in\A$ and $P\in\P$, we have that
\begin{flalign*}
\Psi(\X_{M\oplus P[1]})&=\mu\circ(\int\otimes\int)\circ\Delta(\X_{M\oplus P[1]})\\
&=\mu\circ(\int\otimes\int)(\sum\limits_{[U],[E]}q^{\lr{\hat{M}-\hat{E},\hat{E}-\hat{P}}}F_{UE}^M\X_U\otimes\X_{E\oplus P[1]})\\
&=\mu(\sum\limits_{[U],[E]}q^{\lr{\mathbf{m}-\mathbf{e},\mathbf{e}-\mathbf{p}}}F_{UE}^MX^{\bar{\mathbf{u}}}\otimes X^{\bar{\mathbf{e}}-\bar{\mathbf{p}}})\\
&=\sum\limits_{[U],[E]}v^{2\lr{\mathbf{m}-\mathbf{e},\mathbf{e}-\mathbf{p}}}F_{UE}^M\mu(X^{\bar{\mathbf{m}}-\bar{\mathbf{e}}}\otimes X^{\bar{\mathbf{e}}-\bar{\mathbf{p}}})\\
&=\sum\limits_{\mathbf{e}}v^{2\lr{\mathbf{m}-\mathbf{e},\mathbf{e}-\mathbf{p}}}|\mathrm{Gr}_{\mathbf{e}}M|\mu(X^{\bar{\mathbf{m}}-\bar{\mathbf{e}}}\otimes X^{\bar{\mathbf{e}}-\bar{\mathbf{p}}})\end{flalign*}
\begin{flalign*}&=\sum\limits_{\mathbf{e}}v^{\lr{\mathbf{m}-\mathbf{e},\mathbf{e}-\mathbf{p}}-(\mathbf{m}-\mathbf{e},\mathbf{e}-\mathbf{p})}|\mathrm{Gr}_{\mathbf{e}}M|
X^{-\widetilde{E}(\mathbf{m}-\mathbf{e})-\widetilde{E}'(\mathbf{e}-\mathbf{p})}\\
&=\sum\limits_{\mathbf{e}}v^{\lr{\mathbf{p}-\mathbf{e},\mathbf{m}-\mathbf{e}}}|\mathrm{Gr}_{\mathbf{e}}M|
X^{\widetilde{E}'(\mathbf{p}-\mathbf{e})-\widetilde{E}(\mathbf{m}-\mathbf{e})}.
\end{flalign*}
\end{proof}

The quantum Caldero--Chapoton map associated with a finite acyclic valued quiver has been defined in \cite{Rupel1} and \cite{Qin}.
In \cite{Rupel1}, the quantum Caldero--Chapoton map is defined for
$kQ$-modules, while in \cite{Qin} such a map is defined for coefficient-free rigid objects
in the cluster category associated with $Q$. Let $M\in\A$ and $P\in\P$. We recall from \cite{DXZ} that the {\em quantum cluster character} associated to $M\oplus P[1]$ is defined to be $$X_{M\oplus P[1]}=\sum\limits_{\mathbf{e}}v^{\lr{\mathbf{p}-\mathbf{e},\mathbf{m}-\mathbf{e}}}|\mathrm{Gr}_{\mathbf{e}}M|
X^{-\widetilde{B}\mathbf{e}-\widetilde{E}\mathbf{m}+{\bf t}_P},$$
where ${\bf t}_P=\Dim (P/\rad P)$.
\begin{lemma}\label{assumption}
Let $M\in\A$ and $P\in\P$. If $\widetilde{E}'\mathbf{p}={\bf t}_P$, then $\Psi(\X_{M\oplus P[1]})=X_{M\oplus P[1]}$.
\end{lemma}
\begin{proof}
Since $\widetilde{B}=\widetilde{E}'-\widetilde{E}$ and $\widetilde{E}'\mathbf{p}={\bf t}_P$, we obtain that
$$\widetilde{E}'(\mathbf{p}-\mathbf{e})-\widetilde{E}(\mathbf{m}-\mathbf{e})=-\widetilde{B}\mathbf{e}-\widetilde{E}\mathbf{m}+{\bf t}_P.$$
Thus, by Theorem \ref{main result}, we get that $\Psi(\X_{M\oplus P[1]})=X_{M\oplus P[1]}$.
\end{proof}
\begin{remark}\label{zhuxishu}
If the quiver $\widetilde{Q}$ satisfies that there is no arrow from $i$ to $j$ for any vertices $1\leq i\leq n$ and $n<j\leq 2n$, i.e., for each $1\leq i\leq n$ the indecomposable projective ${\mathfrak{S}}$-module associated to $i$ is also a projective $\widetilde{\mathfrak{S}}$-module. Then $\widetilde{E}'\mathbf{p}={\bf t}_P$ for any $P\in\P$, thus $\Psi(\X_{M\oplus P[1]})=X_{M\oplus P[1]}$. In particular, we consider the quiver $\widetilde{Q}$ associated to $Q$ as follows: for each vertex $1\leq i\leq n$ we add the arrow $n+i\longrightarrow i$. In this case, by \cite{Rupel2} there exists a compatible pair $(\Lambda,\widetilde{B})$.
\end{remark}

In what follows, let $(\Lambda,\widetilde{B})$ be a compatible pair as mentioned in Remark \ref{zhuxishu}. Noting that the quantum torus $\mathcal{T}_\Lambda$ is an Ore domain, we denote by $\mathcal {F}$ its skew-field of fractions.
Let $\A(\Lambda,\widetilde{B})$ be the {\em quantum cluster algebra} (cf. \cite{BZ05}) associated to the compatible pair $(\Lambda,\widetilde{B})$, which is defined to be the $\mathbb{Z}[v,v^{-1}]$-subalgebra of $\mathcal {F}$ generated by all quantum cluster variables and coefficients. This quantum cluster algebra $\A(\Lambda,\widetilde{B})$ is called the {\em principal coefficient quantum cluster algebra}.  We also denote $\A(\Lambda,\widetilde{B})$ by $\A_q(Q)$. Let $\A\H_q(Q)$ be the $\mathbb{Z}[v,v^{-1}]$-subalgebra of $\mathcal {F}$ generated by all the quantum cluster characters $X_{M\oplus P[1]}$ and coefficients $X^{\widetilde{\alpha}}$, where $M\in\A$, $P\in\P$ and $\alpha\in\mathbb{Z}^n$.

By Theorem \ref{main result}, we recover the following
\begin{corollary}\cite[Theorem 8.4]{DXZ}
There exists a surjective algebra homomorphism $$\xymatrix{\Psi: \mathcal{MH}_{\Lambda}(\A)\ar@{->>}[r]& \mathcal{AH}_{q}(Q)}$$ defined on generators by
$$K_{\alpha}\mapsto X^{\widetilde{\alpha}}\quad\text{and}\quad\mathbb{X}_{M\oplus P[1]}\mapsto X_{M\oplus P[1]}$$
for any $\alpha\in\mathbb{Z}^n$, $M\in\A$ and $P\in\P$.
\end{corollary}

According to \cite{Rupel2}, it is known that the quantum cluster algebra $\A_q(Q)$ is the subalgebra of $\mathcal{AH}_{q}(Q)$ generated by
\begin{equation*}
\{X^{\widetilde{\alpha}}, X_{M\oplus P[1]}~|~\alpha\in\mathbb{Z}^n, M\in\Ind(\A)~~\text{is rigid and}~~P\in\Ind(\P)\}.
\end{equation*}
Parallelly, we define $\mathcal{MC}_{\Lambda}(\A)$ to be the subalgebra of $\mathcal{MH}_{\Lambda}(\A)$ generated by $$\{K_{\alpha},\X_{M\oplus P[1]}~|~\alpha\in\mathbb{Z}^n, M\in\Ind(\A)~~\text{is rigid and}~~P\in\Ind(\P)\}.$$
Then we have the following

\begin{corollary}\cite[Corollary]{DXZ}
There exists a surjective algebra homomorphism $$\xymatrix{\pi: \mathcal{MC}_{\Lambda}(\A_Q)\ar@{->>}[r]& \mathcal{A}_{q}(Q).}$$\end{corollary}

Using Theorem \ref{main result}, we can also recover the cluster multiplication formulas (cf. \cite{DX,Fei,DSC,DXZ}).
\begin{corollary}\label{cfgs}
Let $M,N\in\A$ and $P\in\P$. Then
\vspace{0.2cm}
\begin{itemize}
  \item [(1)] ${X}_{M}{X}_{N}=q^{\frac{1}{2}\Lambda(\widetilde{E}{\bf m},
\widetilde{E}\bf{n})+\langle\bf{m},\bf{n}\rangle}\sum\limits_{[L]}\frac{|\mathrm{Ext}_{\A}^{1}(M,N)_{L}|}{|\mathrm{Hom}_{\A}(M,N)|}X_L;$
\vspace{0.4cm}
\item [(2)] $\ X_{P[1]}X_{M}=q^{\frac{1}{2}\Lambda(\widetilde{E}{\bf m},\widetilde{E}{\bf p})-\lr{{\bf p},{\bf m}}}
\sum\limits_{[B],[Q]}|{}_Q\Hom_{\A}(P,M)_B| X_{B\oplus Q[1]}.$
\end{itemize}
\end{corollary}
\begin{proof}
It is proved by (\ref{x4}) and (\ref{x6}) in Proposition \ref{tdygx} together with Theorem \ref{main result} and Lemma \ref{xq0}.
\end{proof}
\begin{remark}
We obtain more than relations given in Corollary \ref{cfgs}. In fact,
each relation in Proposition \ref{tdygx} corresponds to a relation in $\mathcal{AH}_{q}(Q)$.
\end{remark}

\section{An interpretation of $g$-vectors}
Let $(\Lambda,\widetilde{B})$ be a compatible pair as mentioned in Remark \ref{zhuxishu}. In particular, the quantum cluster algebra $\A(\Lambda,\widetilde{B})$ is of principal coefficients.  We endow the quantum torus $\mathcal{T}_\Lambda$ a $\mathbb{Z}^n$-grading by setting
\[\deg X^{\alpha}:=\sum_{i=1}^n a_ie_i-\sum_{j=1}^na_{n+j}\beta_j,\]
where $\alpha=(a_1,\cdots, a_{2n})^{\rm tr}\in \mathbb{Z}^{2n}$ and $\beta_j$ is the $j$-th column vector of the matrix $B$. According to \cite[Proposition 6.1]{FZ4}, every quantum cluster variable $z$ of $\A(\Lambda,\widetilde{B})$ is homogeneous with respect to the above $\mathbb{Z}^n$-grading and the {\it $g$-vector} $g(z)$ of $z$ is defined as the degree of $z$.
Note that the non-frozen quantum cluster variables are precisely
\[\{X_{M}, X_{P[1]}~|~M\in\Ind(\A)~~\text{is rigid and}~~P\in\Ind(\P)\}.\]
We are going to give a categorical interpretation of $g$-vectors via the morphism category $C_2(\P)$.
Similar to Lemma \ref{ftg}, $\hat{C}_{P_1},\cdots, \hat{C}_{P_n},\hat{K}_{P_1},\cdots, \hat{K}_{P_n}$ is also a $\mathbb{Z}$-basis of $K(C_2(\P))$. For each $M_\bullet \in C_2(\P)$, let
\[\operatorname{ind} (M_\bullet)=(h_1,\cdots, h_{2n})^{\rm tr}\in \mathbb{Z}^{2n}\]
be the coordinate vector of $\hat{M}_\bullet$ with respect to the basis $\hat{C}_{P_1},\cdots, \hat{C}_{P_n},\hat{K}_{P_1},\cdots, \hat{K}_{P_n}$. We define
\[\operatorname{ind}^\circ(M_\bullet)=(h_1,\cdots, h_n)^{\rm tr}\in \mathbb{Z}^n\]
to be the truncation of $\operatorname{ind}(M_\bullet)$.
\begin{proposition}
Let $M\in\Ind(\A)$ be rigid and $P\in\Ind(\P)$, we have that
\[g(X_{M})=-\operatorname{ind}^\circ(C_M)~ \text{and}~ g(X_{P[1]})=-\operatorname{ind}^\circ(Z_P).\]
\end{proposition}
\begin{proof} By the definition of $X_M$, we have that
\[g(X_M)=\deg X_M=\deg X^{-\widetilde{E}{\bf m}}=-E{\bf m}-B{\bf m}=-E'{\bf m},\] where $E, E'$ denote the principal parts of the matrices $\widetilde{E}$ and $\widetilde{E}'$, respectively.

Let $0\to \bigoplus_{i=1}^nb_iP_i\to \bigoplus_{i=1}^n a_iP_i\to M\to 0$ be a projective resolution of $M$. It follows that the $k$-th component of $-E'{\bf m}$ is $-\langle M, S_k\rangle=-(a_k-b_k)$. Consequently, $g(X_M)=-(a_1-b_1,\cdots, a_n-b_n)^{\rm tr}$.
On the other hand, by Lemma \ref{resol}, we have the short exact sequence
\[0\to \bigoplus_{i=1}^n b_iC_{P_i}\to \bigoplus_{i=1}^n a_iC_{P_i}\oplus \bigoplus_{i=1}^n b_iK_{P_i}\to C_M\to 0.\]
In particular, $\operatorname{ind}^\circ(C_M)=(a_1-b_1,\cdots,a_n-b_n)^{\rm tr}=-g(X_M)$.

Without loss of generality, we assume that $P=P_k$ for some $1\leq k\leq n$. By definition, we have $g(X_{P[1]})=e_k$. Again by Lemma \ref{resol}, we have the short exact sequence
\[0\to C_{P_k}\to K_{P_k}\to Z_{P_k}\to 0.\]
Hence, $g(X_{P_k[1]})=-\operatorname{ind}^0(Z_{P_k})$.
\end{proof}

\section{Appendix: Quantum cluster characters of a derived Hall subalgebra}
\subsection{Derived Hall algebras}
The derived Hall algebra of the bounded derived category $D^b(\A)$ of $\A$ was introduced in \cite{Toen2006} (see also \cite{XiaoXu}). By definition, the (Drinfeld dual) {\em derived Hall algebra} $\mathcal {D}\mathcal {H}(\A)$ is the free $\ZZ[v,v^{-1}]$-module with the basis $\{u_{X_\bullet}~|~X_\bullet\in \Iso(D^b(\A))\}$ and the multiplication defined by
\begin{equation}
u_{X_\bullet}\diamond u_{Y_\bullet}=\sum\limits_{[{Z_\bullet}]}\frac{|\Ext^1_{D^b(\A)}({X_\bullet},{Y_\bullet})_{Z_\bullet}|}{\prod\limits_{i\geq0}|\Hom_{D^b(\A)}({X_\bullet}[i],{Y_\bullet})|^{(-1)^i}} u_{Z_\bullet},
\end{equation}
where $\Ext^1_{D^b(\A)}({X_\bullet},{Y_\bullet})_{Z_\bullet}$ is defined to be $\Hom_{D^b(\A)}({X_\bullet},{Y_\bullet}[1])_{{Z_\bullet}[1]}$, which denotes the subset of $\Hom_{D^b(\A)}({X_\bullet},{Y_\bullet}[1])$ consisting of morphisms $f:{X_\bullet}\rightarrow {Y_\bullet}[1]$ whose cone is isomorphic to ${Z_\bullet}[1]$.

For any ${X_\bullet},{Y_\bullet}\in D^b(\A)$, define
\begin{equation*}
\lr{{X_\bullet},{Y_\bullet}}:=\sum\limits_{i\in\mathbb{Z}}(-1)^i\dim_k\Hom_{D^b(\A)}({X_\bullet},{Y_\bullet}[i]),
\end{equation*}
it also descends to give a bilinear form on the Grothendieck group of $D^b(\A)$. Moreover, it coincides with the Euler form of $K(\A)$ over the objects in $\A$. In particular, for any $M,N\in\A$ and $i,j\in\mathbb{Z}$, we have that $\lr{M[i],N[j]}=(-1)^{i-j}\lr{M,N}$.

Let us twist the multiplication in $\mathcal {D}\mathcal {H}(\A)$ as follows:
\begin{equation}u_{X_\bullet}\ast u_{Y_\bullet}=q^{\lr{{X_\bullet},{Y_\bullet}}} u_{X_\bullet}\diamond u_{Y_\bullet}\end{equation}
for any ${X_\bullet},{Y_\bullet}\in D^b(\A)$.
The \emph{twisted derived Hall algebra} $\mathcal {D}\mathcal {H}_{q}(\A)$ is the same module as $\mathcal {D}\mathcal {H}(\A)$, but with the twisted multiplication. Then we have the following
\begin{proposition}{\rm(\cite{Toen2006})}\label{twistderived}
$\mathcal {D}\mathcal {H}_q(\A)$ is an associative unital algebra generated by the elements in $\{u_{M[i]}~|~M\in\Iso(\A),~i\in \mathbb{Z}\}$ and the following relations
\begin{flalign}
&u_{M[i]}\ast u_{N[i]}=q^{\lr{M,N}}\sum\limits_{[L]}{\frac{{|\Ext_\mathcal{A}^1{{(M,N)}_L}|}}{{|\Hom_\mathcal{A}(M,N)|}}}u_{L[i]};\\
&u_{M[i+1]}\ast u_{N[i]}=q^{-\lr{M,N}}\sum\limits_{[X],[Y]}|{}_X\Hom_{\A}(M,N)_Y| u_{Y[i]}\ast u_{X[i+1]};\\
&u_{M[i]}\ast u_{N[j]}=q^{(-1)^{i-j}\lr{M,N}} u_{N[j]}\ast u_{M[i]}, \quad i-j>1.
\end{flalign}
\end{proposition}

Let $C_{\A}$ be the subcategory of $D^b(\A)$ consisting of objects $M\oplus P[1]$ with $M\in\A,P\in\P$.
Let $\mathcal {D}\mathcal {H}_q^c(\A)$ be the submodule of $\mathcal {D}\mathcal {H}_q(\A)$ spanned by all elements $u_{X_\bullet}$ with $X_\bullet\in C_{\A}$.
\begin{lemma}~\label{l:extension-closed}
The submodule $\mathcal {D}\mathcal {H}_q^c(\A)$ is a subalgebra of $\mathcal {D}\mathcal {H}_q(\A)$.
\end{lemma}
\begin{proof}
We only need to prove that the subcategory $C_{\A}$ is closed under extensions. Let
\begin{equation}\label{zhengy}
M[-1]\oplus P\rightarrow N\oplus Q[1]\rightarrow X_\bullet\rightarrow M\oplus P[1]\end{equation}
be any triangle in $D^b(\A)$ with $M,N\in\A$ and $P,Q\in\P$. By considering the long exact sequence of homology groups induced by (\ref{zhengy}), we obtain that the $i$-th homology group of $X_\bullet$ vanishes unless $i=-1,0$. Hence, we obtain that $X_\bullet\cong X\oplus Y[1]$ in $D^b(\A)$ for some $X,Y\in\A$, we need to prove that $Y\in\P$.

Since $\Hom(M[-1]\oplus P,Q[1])=0$, by \cite[Lemma 2.5]{PX}, we know that $Q$ is a direct summand of $Y$. Let $Y\cong Q\oplus Y'$ and $M[-1]\oplus P\rightarrow N\rightarrow X\oplus Y'[1]\rightarrow M\oplus P[1]$ be a triangle, then we have a long exact sequence in homologies
$$0\rightarrow Y'\longrightarrow P\longrightarrow N\longrightarrow X\longrightarrow M\rightarrow0.$$ Since $\A$ is hereditary, we get that $Y'$ is projective. Thus, $Y$ is projective and $X_\bullet\in C_{\A}$.
\end{proof}

Using Proposition \ref{twistderived}, we obtain the following
\begin{proposition}\label{dzscgx}
The subalgebra $\mathcal {D}\mathcal {H}_q^c(\A)$ is generated by the elements in $\{u_{M}, u_{P[1]}~|~M\in\A, P\in\P\}$, and the following relations
\begin{flalign}
&u_{P[1]}\ast u_{Q[1]}=
u_{(P\oplus Q)[1]}
=u_{Q[1]}\ast u_{P[1]};\\&
u_{M}\ast u_{N}=q^{\lr{M,N}}\sum_{[L]}\frac{|\mathrm{Ext}_{\A}^{1}(M,N)_{L}|}{|\mathrm{Hom}_{\A}(M,N)|}u_L;\\
&u_{M}\ast u_{P[1]}=u_{M\oplus
P[1]};\\
&u_{P[1]}\ast u_{M}=q^{-\lr{P,M}}
\sum\limits_{[B],[Q]}|{}_Q\Hom_{\A}(P,M)_B|u_{B\oplus Q[1]};\end{flalign}
for any $M,N\in\A$ and $P,Q\in\P$.
\end{proposition}
\begin{proof}
By Proposition \ref{twistderived}, we only need to note that for any $M, N\in\A$ and $i\in\mathbb{Z}$, the following identity
\begin{equation}u_{M[i]}\ast u_{N[i+1]}=u_{M[i]\oplus N[i+1]}\end{equation}
holds in $\mathcal {D}\mathcal {H}_q(\A)$.
\end{proof}

Combining Proposition \ref{dzscgx} with Theorem \ref{ydygx}, we obtain the following
\begin{corollary}\label{dhtth}
There exists an embedding of algebras
$$\xymatrix{\varphi:\mathcal {D}\mathcal {H}_q^c(\A)\ar@{^{(}->}[r]&\M\H(\A)}$$
defined on generators by $u_M\mapsto\X_M$ and $u_{P[1]}\mapsto\X_{P[1]}.$
\end{corollary}

The comultiplication $\Delta$ on $\M\H(\A)$ defined in (\ref{yucheng}) induces a comultiplication
$$\Delta:\mathcal {D}\mathcal {H}_q^{c}(\A)\longrightarrow \mathcal {D}\mathcal {H}_q^{c}(\A)\otimes \mathcal {D}\mathcal {H}_q^{c}(\A)$$
defined by
\begin{equation}\label{yucheng2}\Delta(u_{L\oplus P[1]}):=\sum\limits_{[M],[N]}
q^{\lr{\hat{M},\hat{N}-\hat{P}}}F_{MN}^L(u_M\otimes u_{N\oplus P[1]})
\end{equation} for any $L\in\A$ and $P\in\P$.
Let us define the multiplication $\ast$ on $\mathcal {D}\mathcal {H}_q^{c}(\A)\otimes\mathcal {D}\mathcal {H}_q^{c}(\A)$ by
\begin{equation}\label{xingc}
\begin{split}
&(u_{M\oplus P[1]}\otimes u_{N\oplus Q[1]})\ast(u_{U\oplus S[1]}\otimes u_{V\oplus T[1]}):=\\
&q^{(\hat{N}-\hat{Q},\hat{U}-\hat{S})+\lr{\hat{M}-\hat{P},\hat{V}-\hat{T}}}
(u_{M\oplus P[1]}\ast u_{U\oplus S[1]}\otimes u_{N\oplus Q[1]}\ast u_{V\oplus T[1]})
\end{split}
\end{equation} for any $M, N, U, V\in\A$ and $P, Q, S, T\in\P$.
In a similar way to Proposition \ref{sdsjg}, we obtain the following
\begin{proposition}
The map $\Delta:(\mathcal {D}\mathcal {H}_q^{c}(\A),\ast)\longrightarrow (\mathcal {D}\mathcal {H}_q^{c}(\A)\otimes\mathcal {D}\mathcal {H}_q^{c}(\A),\ast)$ is a homomorphism of algebras.
\end{proposition}
\begin{remark}
Let $D^{[-1,0]}(\A)$ be the subcategory of $D^b(\A)$ consisting of objects $L\oplus X[1]$ with $L,X\in\A$.
Let $\mathcal {D}\mathcal {H}_q^{[-1,0]}(\A)$ be the submodule of $\mathcal {D}\mathcal {H}_q(\A)$ spanned by all elements $u_{X_\bullet}$ with $X_\bullet\in D^{[-1,0]}(\A)$. Then $\mathcal {D}\mathcal {H}_q^{[-1,0]}(\A)$ is a subalgebra of $\mathcal {D}\mathcal {H}_q(\A)$, which is related to the {\em Heisenberg double Hall algebra} defined in \cite{Kap}. Let $$\Delta:\mathcal {D}\mathcal {H}_q^{[-1,0]}(\A)\longrightarrow\mathcal {D}\mathcal {H}_q^{[-1,0]}(\A)\otimes\mathcal {D}\mathcal {H}_q^{[-1,0]}(\A)$$ be the homomorphism of modules defined by
\begin{equation}\Delta(u_{L\oplus X[1]}):=\sum\limits_{[M],[N]}
q^{\lr{\hat{M},\hat{N}-\hat{X}}}F_{MN}^L(u_M\otimes u_{N\oplus X[1]})
\end{equation} for any $L, X\in\A$. Consider the multiplication $\ast$ which is similar to (\ref{xingc}) on $\mathcal {D}\mathcal {H}_q^{[-1,0]}(\A)\otimes\mathcal {D}\mathcal {H}_q^{[-1,0]}(\A)$. In a similar way to the proof of Proposition \ref{sdsjg}, we can obtain that the map $\Delta:(\mathcal {D}\mathcal {H}_q^{[-1,0]}(\A),\ast)\longrightarrow(\mathcal {D}\mathcal {H}_q^{[-1,0]}(\A)\otimes\mathcal {D}\mathcal {H}_q^{[-1,0]}(\A),\ast)$ is also a homomorphism of algebras.
\end{remark}
\subsection{Integration map on the derived Hall subalgebra $\mathcal {D}\mathcal {H}_q^c(\A)$}
Let $Q$, $\widetilde{Q}$ be the same as given in Section 5. Let $\A$ be the category of finite dimensional left $\mathfrak{S}$-modules. In this subsection, we define an integration map on the derived Hall subalgebra $\mathcal {D}\mathcal {H}_q^c(\A)$. For each positive integer $t$,
let $\mathcal{T}_t$ be the $\ZZ[v,v^{-1}]$-algebra with a basis $\{X^{\alpha}~|~\alpha\in \mathbb{Z}^t\}$ and
multiplication defined by
\[X^{\alpha}\diamond X^{\beta}=X^{\alpha+\beta}.\]
It is well known that there is an isomorphism of groups
$$f: K(D^b(\A))\longrightarrow K(\A)$$ defined by $f(\hat{X}_\bullet)=\sum\limits_{i\in\mathbb{Z}}(-1)^i\Dim X_i=:\Dim {X}_\bullet.$ Moreover, $\lr{X_\bullet,Y_\bullet}=\lr{\Dim {X}_\bullet,\Dim {Y}_\bullet}.$

\begin{lemma}\label{ceuler}
For any objects $X_\bullet, Y_\bullet\in C_{\A}$, we have that
$\dim_k\Hom_{D^b(\A)}(X_\bullet,Y_\bullet[i])=0$
if $|i|>1$.
\end{lemma}
\begin{proof}
Let $X_\bullet=M\oplus P[1]$ and $Y_\bullet=N\oplus Q[1]$ with $M,N\in\A$ and $P,Q\in\P$.

If $i>1$, since $\A$ is hereditary and $P$ is projective, we obtain that \begin{equation*}\begin{split}\Hom_{D^b(\A)}(X_\bullet,Y_\bullet[i])&=\Hom_{D^b(\A)}(M\oplus P[1],N[i]\oplus Q[i+1])\\&\cong\Hom_{D^b(\A)}(P[1],N[i])\\&=0.\end{split}\end{equation*}

For $i<-1$, it is easy to see that
$\Hom_{D^b(\A)}(X_\bullet,Y_\bullet[i])=0.$
\end{proof}

\begin{proposition}\label{jfys}
The integration map $$\int:\mathcal {D}\mathcal {H}_q^c(\A)\longrightarrow\mathcal{T}_n,~~u_{X_\bullet}\mapsto X^{{\bf dim\,} {X}_\bullet}$$
is a homomorphism of algebras.
\end{proposition}
\begin{proof} For any objects $X_\bullet, Y_\bullet\in C_{\A}$,
\begin{flalign*}\int u_{X_\bullet}\ast u_{Y_\bullet} &= q^{\lr{X_\bullet,Y_\bullet}}\sum\limits_{[{Z_\bullet}]}\frac{|\Ext^1_{D^b(\A)}({X_\bullet},{Y_\bullet})_{Z_\bullet}|}{\prod\limits_{i\geq0}|\Hom_{D^b(\A)}({X_\bullet}[i],{Y_\bullet})|^{(-1)^i}} X^{{\bf dim\,} Z_\bullet}\\
&=q^{\sum\limits_{i>0}(-1)^i{\rm dim}_k{\rm Hom}_{D^b(\A)}({X_\bullet},{Y_\bullet}[i])}\sum\limits_{[Z_\bullet]} |\Ext^1_{D^b(\A)}({X_\bullet},{Y_\bullet})_{Z_\bullet}| X^{{\bf dim\,} X_\bullet+{\bf dim\,} Y_\bullet}\\
&=X^{{\bf dim\,} X_\bullet+{\bf dim\,} Y_\bullet}\quad\quad\text{(by~Lemma~\ref{ceuler})}\\
&=X^{{\bf dim\,} X_\bullet}\diamond X^{{\bf dim\,} Y_\bullet}\\
&=\int u_{X_\bullet}\diamond\int u_{Y_\bullet}.\end{flalign*}
\end{proof}

\subsection{Quantum cluster characters via derived Hall subalgebras}
Let $Q, \widetilde{Q}$ be the same as given in Section~$5$. We keep the notations as in Section~$5$, in particular, we have $m\times n$ integral matrices $\widetilde{B}$ and $\widetilde{E}$. Let $\mathcal{A}$ (resp. $\widetilde{\mathcal{A}}$) be the category of finite dimensional left $\mathfrak{S}$ (resp. $\mathfrak{\widetilde{S}}$)-modules. We may identify $\mathcal{A}$ with the full subcategory of $\widetilde{\mathcal{A}}$ consisting of modules with supports on $Q$. For an $\mathfrak{S}$-module $X$ we also denote by $\bf{x}$ the dimension vector of $X$ viewed as an $\widetilde{\mathfrak{S}}$-module, since this should not cause confusion by the context. Let $R(\widetilde{Q})$ and $R'(\widetilde{Q})$ be the $m\times m$ matrices with the $i$-th row and $j$-th column elements given respectively by
$$r_{ij}=\dim_{\mathcal {D}_i}\Ext_{\widetilde{\mathfrak{S}}}^1(S_j,S_i)$$
and
$$r'_{ij}=\dim_{{\mathcal {D}_i}^{op}}\Ext_{\widetilde{\mathfrak{S}}}^1(S_i,S_j),$$
where $1\leq i,j\leq m$. Define $B(\widetilde{Q})=R'(\widetilde{Q})-R(\widetilde{Q})$,
$E(\widetilde{Q})=I_m-R'(\widetilde{Q})$ and $E'(\widetilde{Q})=I_m-R(\widetilde{Q})$.
Note that $\widetilde{B}$ is the submatrix of $B(\widetilde{Q})$ consisting of the first $n$ columns.

In what follows, we assume that there is a skew-symmetric $m\times m$ integral matrix $\Lambda$ such that \begin{equation}-\Lambda B(\widetilde{Q})=\operatorname{diag}\{d_1,\cdots, d_m\}.\end{equation}

As in Section $6$, we twist the multiplication on $\mathcal {D}\mathcal {H}_q^c(\widetilde{\A})$, and define $\mathcal {D}\mathcal {H}_\Lambda^c(\widetilde{\A})$ to be the same module as $\mathcal {D}\mathcal {H}_q^c(\widetilde{\A})$ but with the twisted
multiplication defined on basis elements by
{\begin{equation}
\begin{split}
   u_{M\oplus P[1]}\star u_{N\oplus Q[1]}:=
   v^{\Lambda(E'(\widetilde{Q})({\bf m}-{\bf p}),E'(\widetilde{Q})(\bf{n}
   -\bf{q}))}
   u_{M\oplus P[1]}\ast u_{N\oplus Q[1]},\end{split}
\end{equation}}
where $M, N\in\widetilde{\A}$ and $P, Q\in\P_{\widetilde{\A}}$. We also twist the multiplication on the tensor algebra $(\mathcal {D}\mathcal {H}_q^c(\widetilde{\A})\otimes\mathcal {D}\mathcal {H}_q^c(\widetilde{\A}),\ast)$ by defining
\begin{equation*}\begin{split}(u_{M\oplus P[1]}\otimes u_{N\oplus Q[1]})\star(u_{U\oplus S[1]}\otimes u_{V\oplus T[1]}):=
v^{\lambda}
(u_{M\oplus P[1]}\otimes u_{N\oplus Q[1]})\ast(u_{U\oplus S[1]}\otimes u_{V\oplus T[1]}),
\end{split}
\end{equation*}where $\lambda=\Lambda(E'(\widetilde{Q})({\bf m}-{\bf p}+{\bf n}-{\bf q}),E'(\widetilde{Q})(\bf{u}-\bf{s}+\bf{v}-\bf{t}))$, $M, N,U,V\in\widetilde{\A}$ and $P, Q,S,T\in\P_{\widetilde{\A}}$. In a similar way to Lemma \ref{ts1}, we obtain the following
\begin{lemma}
The map $\Delta:(\mathcal {D}\mathcal {H}_\Lambda^c(\widetilde{\A}),\star)\longrightarrow (\mathcal {D}\mathcal {H}_q^c(\widetilde{\A})\otimes\mathcal {D}\mathcal {H}_q^c(\widetilde{\A}),\star)$ is a homomorphism of algebras.
\end{lemma}
We twist the multiplication on the tensor algebra of the torus $\mathcal{T}_m$ by defining
\begin{equation}(X^\alpha\otimes X^\beta)\star (X^\gamma\otimes X^\delta):=q^{\frac{1}{2}\Lambda(E'(\widetilde{Q})(\alpha+\beta), E'(\widetilde{Q})(\gamma+\delta))+(\beta, \gamma)+\langle \alpha, \delta\rangle}X^{\alpha+\gamma}\otimes X^{\beta+\delta}
\end{equation} for any $\alpha, \beta, \gamma, \delta\in\mathbb{Z}^m$.
In a similar way to Lemma \ref{ts2}, we obtain the following
\begin{lemma}
The map $\int\otimes\int: (\mathcal {D}\mathcal {H}_q^c(\widetilde{\A})\otimes\mathcal {D}\mathcal {H}_q^c(\widetilde{\A}),\star)\longrightarrow(\mathcal{T}_m\otimes\mathcal{T}_m,\star)$ is a homomorphism of algebras.
\end{lemma}
In a similar way to Proposition \ref{ts3}, we obtain the following
\begin{proposition}
The map $\mu: (\mathcal{T}_m\otimes\mathcal{T}_m,\star)\longrightarrow(\mathcal{T}_\Lambda,\star)$ defined by
$$\mu(X^{\alpha}\otimes X^{\beta})=v^{-(\alpha,\beta)-\lr{\alpha,\beta}}X^{-{E}(\widetilde{Q})\alpha-{E}'(\widetilde{Q})\beta},$$
where $\alpha,\beta\in\mathbb{Z}^m$, is a homomorphism of algebras.
\end{proposition}

Define $\psi:=\mu\circ(\int\otimes\int)\circ\Delta$. In other words, we have the following commutative diagram
\begin{equation}
\xymatrix{(\mathcal {D}\mathcal {H}_\Lambda^c(\widetilde{\A}),\star)\ar[r]^-{\psi}\ar[d]_-{\Delta}&(\mathcal{T}_\Lambda,\star)\\(\mathcal {D}\mathcal {H}_q^c(\widetilde{\A})\otimes\mathcal {D}\mathcal {H}_q^c(\widetilde{\A}),\star)\ar[r]^-{\int\otimes\int}&(\mathcal{T}_m\otimes\mathcal{T}_m,\star).\ar[u]_-{\mu}}
\end{equation}
In a similar way to Theorem \ref{main result}, we obtain the following
\begin{theorem}
The map $\psi:(\mathcal {D}\mathcal {H}_\Lambda^c(\widetilde{\A}),\star)\longrightarrow(\mathcal{T}_\Lambda,\star)$ is a homomorphism of algebras. Moreover, for any $M\in\widetilde{\A}$ and $P\in\P_{\widetilde{\A}}$,
$$\psi(u_{M\oplus P[1]})=\sum\limits_{\mathbf{e}}v^{\lr{\mathbf{p}-\mathbf{e},\mathbf{m}-\mathbf{e}}}|\mathrm{Gr}_{\mathbf{e}}M|
X^{{E}'(\widetilde{Q})(\mathbf{p}-\mathbf{e})-{E}(\widetilde{Q})(\mathbf{m}-\mathbf{e})}.$$
\end{theorem}
\begin{corollary}\label{mr2t}
For any $M\in\widetilde{\A}$ and $P\in\P_{\widetilde{\A}}$,
$$\psi(u_{M\oplus P[1]})=\sum\limits_{\mathbf{e}}v^{\lr{\mathbf{p}-\mathbf{e},\mathbf{m}-\mathbf{e}}}|\mathrm{Gr}_{\mathbf{e}}M|
X^{-B(\widetilde{Q}){\bf e}-E(\widetilde{Q}){\bf m}+{\bf t_P}}.$$
\end{corollary}
\begin{proof}
We only need to note that ${B}(\widetilde{Q})={E}'(\widetilde{Q})-{E}(\widetilde{Q})$ and ${E}'(\widetilde{Q})\mathbf{p}={\bf t}_P$, thus obtain that
$${E}'(\widetilde{Q})(\mathbf{p}-\mathbf{e})-{E}(\widetilde{Q})(\mathbf{m}-\mathbf{e})=-{B}(\widetilde{Q})\mathbf{e}-{E}(\widetilde{Q})\mathbf{m}+{\bf t}_P.$$\end{proof}
\begin{remark}
In Corollary \ref{mr2t}, we do not need to make the principal coefficient assumption as in Remark \ref{zhuxishu}. In fact, in Corollary \ref{mr2t} we use projectives $P$ in $\widetilde{\A}$ and ${E}'(\widetilde{Q})\mathbf{p}={\bf t}_P$ always holds (cf. \cite[Lemma 1]{Hubery1}), while in Lemma \ref{assumption} we use projectives $P$ in ${\A}$ and the matrix $\widetilde{E}'$ is of size $m\times n$.
\end{remark}

Let $\widetilde{C}_{\A}$ be the full subcategory of $\mathcal{D}^b(\widetilde{\mathcal{A}})$ consisting of objects $M\oplus P[1]$ with $M\in \mathcal{A}$ and $P\in \mathscr{P}_{\widetilde{\mathcal{A}}}$. Let  $\mathcal{DH}_\Lambda^{\tilde{c}}({\mathcal{A}})$ be the submodule of $\mathcal{DH}_\Lambda^c(\widetilde{\mathcal{A}})$ spanned by all elements $u_{X_\bullet}$ with $X_\bullet\in \widetilde{C}_{\A}$.
\begin{lemma}
The submodule $\mathcal{DH}_\Lambda^{\tilde{c}}({\mathcal{A}})$ is a subalgebra of $\mathcal{DH}_\Lambda^c(\widetilde{\mathcal{A}})$.
\end{lemma}
\begin{proof}
We only need to prove that the subcategory $\widetilde{C}_{\A}$ is closed under extensions. Let
\begin{equation}\label{zhengy2}
M[-1]\oplus P\rightarrow N\oplus Q[1]\rightarrow X_\bullet\rightarrow M\oplus P[1]\end{equation}
be any triangle in $D^b(\widetilde{\A})$ with $M,N\in\A$ and $P,Q\in\mathscr{P}_{\widetilde{\mathcal{A}}}$. By Lemma \ref{l:extension-closed}, we assume that $X_\bullet=X\oplus Y[1]$ with $X\in \widetilde{\A}$ and $Y\in\mathscr{P}_{\widetilde{\mathcal{A}}}$, and we need to prove that $X\in\A$.  By considering the long exact
sequence of homology groups induced by (\ref{zhengy2}), we have the following exact sequence
$$\xymatrix{N\ar[r]^-f& X\ar[r]& M\ar[r]& 0.}$$ Thus, we have a short exact sequence $0\rightarrow \im f\rightarrow X\rightarrow M\rightarrow 0.$ Since $\im f$ is a quotient of $N$ and $N\in\A$, we get that $\im f\in\A$. Hence, we obtain that $X\in\A$ since $M\in\A$.
\end{proof}

Let us define $\Psi=\psi\circ l$, where $l$ is the embedding of $\mathcal{DH}_\Lambda^{\tilde{c}}({\mathcal{A}})$ into $\mathcal{DH}_\Lambda^c(\widetilde{\mathcal{A}})$.
\begin{corollary}
The map $\Psi:(\mathcal {D}\mathcal {H}_\Lambda^{\tilde{c}}({\A}),\star)\longrightarrow(\mathcal{T}_\Lambda,\star)$ is a homomorphism of algebras. Moreover, for any $M\in{\A}$ and $P\in\P_{\widetilde{\A}}$,
$$\Psi(u_{M\oplus P[1]})=\sum\limits_{\mathbf{e}}v^{\lr{\mathbf{p}-\mathbf{e},\mathbf{m}-\mathbf{e}}}|\mathrm{Gr}_{\mathbf{e}}M|
X^{-\widetilde{B}{\bf e}-\widetilde{E}{\bf m}+{\bf t}_{{P}}}.$$ That is, $\Psi(u_{M\oplus {P}[1]})$ is precisely the quantum cluster character associated to $u_{M\oplus {P}[1]}$ for the quantum cluster algebra $\mathcal{A}(\Lambda, \widetilde{B})$.
\end{corollary}
\begin{proof}
Clearly, $\Psi$ is a homomorphism of algebras. By Corollary \ref{mr2t}, we only need to note that for any $X\in\A$, we have that
$B(\widetilde{Q}){\bf x}=\widetilde{B}{\bf x}$ and $E(\widetilde{Q}){\bf x}=\widetilde{E}{\bf x}$. Moreover, for any two $\mathfrak{S}$-modules $X$ and $Y$, the Euler forms of $X$ and $Y$ as $\mathfrak{S}$-modules and $\widetilde{\mathfrak{S}}$-modules are equal.
\end{proof}

\section*{Acknowledgments}
The authors would like to thank the anonymous reviewers for the careful reading, helpful comments and suggestions.
This project was partially supported by the National Natural Science Foundation of China (No.s 11971326, 11821001, 11801273, 12271257).

\end{document}